\documentclass[12pt]{amsart}
\usepackage{hyperref}
\hypersetup{pdfauthor={Daniel Burns Jr. and Kin Kwan Leung},pdftitle={The Complex Monge-Amp\`ere Equation, Zoll Metrics and Algebraization}}
\pdfoutput=1

\usepackage{amsmath}
\usepackage{amsfonts}
\usepackage{amsthm}
\usepackage{url}
\usepackage{amssymb}
\usepackage[utf8]{inputenc}
\usepackage{enumerate}
\usepackage{amsmath}
\usepackage{amsthm}
\usepackage{amssymb}
\usepackage{amsfonts}
\usepackage{animate}
\usepackage{amsmath}
\usepackage{amscd}

\usepackage{epsfig,amsthm,enumerate,import}

\theoremstyle{plain}
\newtheorem{theorem}{Theorem}
\newtheorem{prop}[theorem]{Proposition}
\newtheorem{corollary}[theorem]{Corollary}
\newtheorem{lemma}[theorem]{Lemma}
\newtheorem{ques}[theorem]{Question}

\theoremstyle{definition}
\newtheorem{defn}[theorem]{Definition}

\def \r {\mathbb{R}}
\def \z {\mathbb{Z}}
\def \c {\mathbb{C}}

\def \p {\mathbb{P}}

\def \q {\mathcal{Q}}
\def \o {\mathcal{O}}
\def \u {\mathcal{U}}
\def \v {\mathcal{V}}
\def \f {\mathcal{F}}
\def \h {\mathbb{H}}
\def \m {\mathcal{M}}

\def \vs {\vskip 2mm}
\def \d {\partial}
\def \dbar {\bar{\partial}}
\def \ddb {\frac{i}{2\pi}\d\dbar}
\def \dd[#1,#2]{\frac{\partial #1}{\partial #2}}
\def \ddc {dd^c}
\def \D {\triangle}
\def \HH {\mathcal{H}}
\def \and {\; \text{and}\;}
\def \comp {\HH^{k,\alpha}_{z_0, \lambda}(\D, \u_{\lambda})}
\def \compvar[#1,#2]{\HH^{k,\alpha}_{#1, #2}(\D, \u_{#2})}
\def \on {\;\text{on}\;}

\def \for {\; \text{for}\;}
\def \re {\, \text{Re}\,}
\def \im {\,\text{Im}\,}
\def \bX {\bar{X}}
\def \bj {\bar{\jmath}}
\def \bi {\bar{\imath}}

\theoremstyle{remark}
\newtheorem{rmk}[theorem]{Remark}
\numberwithin{theorem}{section}
\makeatletter
\def\v@rt#1#2{\m@th\ooalign{$\hfil#1|\hfil$\crcr$#1#2$}}
\def\captr{\mathrel{\mathpalette\v@rt\cap}}
\makeatother
\begin{document}
\title[Complex Monge-Amp\`ere]{The Complex Monge-Amp\`ere Equation, \\Zoll Metrics and Algebraization}

\author{Daniel Burns Jr.}\address{Department of Mathematics, University of Michigan, Ann Arbor} \email{dburns@umich.edu}
\author{Kin Kwan Leung}\address{Department of Mathematics, University of Toronto} \email{kkleung@math.toronto.edu}
\thanks{Supported in part by NSF grant DMS-1105586.}

\begin{abstract} Let $M$ be a real analytic Riemannian manifold. An adapted complex structure on $TM$ is a complex structure on a neighborhood of the zero section such that the leaves of the Riemann foliation are complex submanifolds. This structure is called entire if it may be extended to the whole of $TM$. We prove here that the only real analytic Zoll metric on the $n$-sphere with an entire adapted complex structure on $TM$ is the round sphere. Using similar ideas, we answer a special case of an algebraization question raised by the first author, characterizing some Stein manifolds as affine algebraic in terms of plurisubharmonic exhaustion functions satisfying the homogeneous complex Monge-Amp\`ere (HCMA) equation. 
\end{abstract}

\maketitle
%
%
%
%
%
%
%%%%%%%%%%%%%%%%%%%%%%%%%%%%%%%%%%%%%%%%%%%%%%%%%%%%%%%%%%%
%%%%%%%%%%%%%%%%%%%%%%%%%%%%%%%%%%%%%%%%%%%%%%%%%%%%%%%%%%%
%																				                 %	
%								      Introduction											         %	
%																				                 %	
%%%%%%%%%%%%%%%%%%%%%%%%%%%%%%%%%%%%%%%%%%%%%%%%%%%%%%%%%%%
%%%%%%%%%%%%%%%%%%%%%%%%%%%%%%%%%%%%%%%%%%%%%%%%%%%%%%%%%%%
%
%
%
%
%
%
\section{Introduction}\label{sec:intro}
In this paper we study two appearances of solutions of the homogeneous complex Monge-Amp\`ere equation
\begin{equation}
	\label{eqn:hcma}
		(i \,\partial\bar{\partial} u)^n = 0 
\end{equation}
in different but related geometric contexts. 
\vs
The first context has to do with the work of LeBrun and Mason \cite{lm} introducing complex analytic techniques into the study of Zoll surfaces. Recall that a Zoll manifold is a Riemannian manifold for which all geodesics are simple and closed, of the same length. LeBrun and Mason \cite{lm} associate to each Zoll metric $g$ on the two sphere $S^2$ a totally real submanifold $N \subset \c\p^2$ diffeomorphic to $\r\p^2$,  and a non-singular quadric curve $\q^1 \subset \c\p^2- N$ ``at infinity", via a kind of twistor transform. $N$ may be interpreted as the moduli space of all unoriented geodesics on $(S^2, g)$, while $\q^1$ may interpreted as the underlying manifold $(S^2, g)$. The complement $\c\p^2 - N$ is foliated by embedded holomorphic disks with good boundary behavior along $N$. In the case of the {\em round} metric $g_0$ on $S^2$, there is an initially similar seeming embedding of $S^2$, this time in the non-singular two-dimensional quadric hypersurface $\q^2 \subset \c\p^3$. This embedding is totally real, $\q^2$ is a projective complexification of $S^2$, and a hyperplane section at infinity gives a non-singular quadric curve $\q^1 \subset \q^2$ which can be interpreted as the moduli space of oriented geodesics on $(S^2, g_0)$. The complement $\q^2 - S^2$ is foliated by holomorphic disks, whose boundaries on $S^2$ trace out the geodesics on $(S^2, g_0)$. This embedding and foliation correspond to an algebraic compactification of the entire Grauert tube ({\em cf.} section \ref{sec:acs} below) associated to the round metric $g_0.$ LeBrun and Mason mention that they first looked at $\q^2$, a double cover of $\c\p^2$ ramified along $\q^1$. Here $\q^2$ is the nonsingular quadric in $\c\p^3$ and $\q^1$ is the intersection of $\q^2$ with a generic hyperplane in $\c\p^3$. The pullback of the totally real $\r\p^2$ is now a 2-sphere. The foliation in the round case is given by the intersection of complexified real planes passing through the origin. Notice that in $\c\p^2$ and $\q^2$, the leaves are holomorphic disks that can be \emph{extended across} the totally real surface to match the opposite leaf. In subsequent work, LeBrun and Mason have shown that this complex presentation of a Zoll metric on $S^2$ can be put to effective use in studying global properties of the family of all Zoll metrics on $S^2$.
\vs
Motivated by a desire to introduce such complex analytic techniques to the study of higher dimensional Zoll manifolds, we tried to exploit the ``dual" description of the round case above via entire Grauert tubes, since the main known methods to study global properties of the family of Zoll surfaces (Radon transforms \cite{guillemin}; twistor transforms \cite{lm}) seem to face serious obstacles to being extended to these cases. Unfortunately, we prove, in Theorem \ref{thm:mainzollthm} below, that in every dimension $n$, including dimension $n$ = 2, the round sphere is the only Zoll manifold diffeomorphic to $S^n$ whose adapted complex structure is entire. This is proved by showing that the tangent bundle of such an $(M, g)$ can be compactified and next identifying the compactification with the complex quadric $\q^n$, via a pretty classification result in algebraic geometry due to Kobayashi-Ochiai \cite{ko} and Kachi-Koll\'ar \cite{kk}. From there it is easy to draw the final conclusion from global geometric properties of $\q^n$. We mention here that Patrizio-Wong have shown that all CROSSes\footnote{``CROSS" equals ``compact rank one symmetric space".} admit entire Grauert tubes with natural projective compactifications. We speculate briefly on this in section \ref{sec:final} below. The other CROSSes admit similar projective, equivariant compactifications which are detailed by Patrizio and Wong \cite{pw}.
\vs
The compactification technique alluded to above proves useful in our second context, namely that of giving potential theoretic characterizations of affine algebraic manifolds among all Stein manifolds. This question goes back, {\em inter al.}, to \cite{gk},  \cite{stoll80}, \cite{burns} and \cite{burns2}. Very interesting work on such a characterization, with a supplementary condition on the curvature of a K\"ahler metric, are due to Demailly \cite{jpd}. Consider the following remark related to this question: given an open Riemann surface $X$ and $\tau: X \to \r^+$ a strictly subharmonic exhaustion function on $X$ for which there exists $\tau_0\in \r^+$, such that $u:=\log \tau$ is harmonic and (necessarily) $du\neq 0$ outside the compact set $K = \{ \tau \leq \tau_0\} \subset X$. This means that there exists $c\in\r$ such that when $u_0\geq c$, $u^{-1}(u_0)$ is a 1-dimensional manifold. Since $u$ is an exhaustion, i.e., is proper, $u^{-1}(u_0)$ must be compact, and hence it must consist of a finite union of circles. The number of circles is independent of $u_0 >> 0$ by Morse theory because $du\neq 0$ for $u_0\geq c$. Varying $u_0\geq c$, we see that $X\cap \{ u\geq c\}$ is a finite disjoint union of cylinders $C_j$. Let $u_j = u|_{C_j}$ and let $u_j^*$ be a harmonic conjugate of $u_j$. The harmonic conjugate is defined up to an additive constant $\alpha_j$, where
\[
\alpha_j = \int_{\d C_j} d^c u.
\]  
To eliminate this constant, we let 
\[
\zeta_j := \exp\left(-(u_j + i u_j^*)\frac{2\pi}{\alpha_j}\right). 
\]
Then $\zeta_j$ is a local holomorphic coordinate on $C_j$ and thus identifies $C_j$ with a punctured disk. This shows that we can compactify $X$ to a compact Riemann surface $\bar{X}$ by filling in the origin of each of the punctured disks $C_j$. Since any compact Riemann surface is projective, $X$ is an affine algebraic variety. We believe this observation was first made by W. Stoll in the eighties, though we were unable to locate it in the literature. The first author asked what will happen when $X$ is an $n$-dimensional manifold, $\tau$ is a strictly plurisubharmonic exhaustion function and $u = \log \tau$ satisfies the homogeneous complex Monge-Amp\`ere equation when $u$ is large: does this guarantee that $X$ is an affine algebraic variety with polynomials the polynomially growing entire functions on $X$, as measured in terms of powers of $\tau$?
\vs
A solution $u$ as above of the HCMA equation defines a foliation of $X$ near infinity by Riemann surfaces. To arrive at a situation resembling Stoll's example in one dimension ``with parameters", we will make the restrictive assumption that this foliation has closed leaves and is locally trivial near infinity (this is equivalent to the Zoll condition in the case of a Grauert tube and its corresponding Monge-Amp\`ere solution). In this case, Theorem \ref{thm:algn} below, we can once again compactify $X$ to $\bar{X}$ by adding a smooth real codimension two manifold $D$ at infinity and, using Kodaira's embedding theorem, prove that $\bar{X}$ is a projective manifold, and $D$ an ample divisor on $\bar{X}$. This implies that $X$ is affine, and the construction shows that polynomial growth with respect to $\tau$ does, indeed, determine the algebraic entire functions on $X$.
\vs
The conditions in Theorem \ref{thm:algn} are restrictive, but in section \ref{sec:algnec} below, however, we will show the converse of this statement, namely, that if $\bar{X}$ is a projective manifold, and $D \subset \bar{X}$ is a smooth, ample divisor, then $X := \bar{X} - D$ admits a strictly plurisubharmonic exhaustion function for which the conditions of Theorem \ref{thm:algn} hold true. Notice that this condition is ``generic" from the point of view of $\bar{X}$ embedded in projective space, namely, the generic hyperplane section would give a triple $(X \subset \bX \supset D := \bar{X} - X)$ satisfying the hypotheses of Theorem \ref{thm:converse}. Thus, Theorems \ref{thm:algn} and \ref{thm:converse} taken together give a characterization of affine $X$ which admit a smooth projective completion and smooth complementary divisor at infinity in terms of p.s.h. solutions of the HCMA equation.

\vs
We now summarize quickly the layout of the paper. In section \ref{sec:acs}, we will review some basic ideas (and notation) introduced in Lempert and Sz\H{o}ke \cite{ls}. In section \ref{sec:entirezoll}, we will use Stoll's idea to show that a real analytic Zoll sphere with entire Grauert tube must be round. In section \ref{sec:algsuff}, we prove the algebraization result, given the existence of a special solution of the HCMA, and in section \ref{sec:algnec} we prove the converse by constructing such a Monge-Amp\`ere solution. In a brief section \ref{sec:final} we list a few open questions.
\vs
Parts of this paper appeared in earlier forms in \cite{burns2}, and in the second author's PhD dissertation \cite{kkl} at the University of Michigan. 
%
%
%
%
%
%
%%%%%%%%%%%%%%%%%%%%%%%%%%%%%%%%%%%%%%%%%%%%%%%%%%%%%%%%%%%
%%%%%%%%%%%%%%%%%%%%%%%%%%%%%%%%%%%%%%%%%%%%%%%%%%%%%%%%%%%
%																				                 %	
%								       Adapted Complex Structures							         %	
%																				                 %	
%%%%%%%%%%%%%%%%%%%%%%%%%%%%%%%%%%%%%%%%%%%%%%%%%%%%%%%%%%%
%%%%%%%%%%%%%%%%%%%%%%%%%%%%%%%%%%%%%%%%%%%%%%%%%%%%%%%%%%%
 %
 %
 %
 %
 %
 %
\section{Adapted complex structure}\label{sec:acs}
In this section, we recall some results and notation from Lempert and Sz\H{o}ke's \cite{ls}.
\vs
Let $(M,g)$ be a compact real-analytic Riemannian manifold with $g$ real analytic, and denote by $TM$ its tangent bundle. Let $\tau \in \r$ and $N_\tau : TM\rightarrow TM$ be the smooth mapping defined by multiplication by $\tau$ in the fibres, i.e. if $x\in M$ and $v\in T_x M$, then $N_\tau (x,v) = (x,\tau v)$. Let $\gamma:\r\rightarrow M$ be an arc-length parametrized geodesic. We define an immersion $\psi_\gamma : \c \rightarrow TM$ 
\begin{equation}
\psi_\gamma (\sigma + i\tau) = N_\tau \dot{\gamma} (\sigma).
\end{equation}
The image of $\psi_\gamma$ is exactly the tangent bundle of $\gamma$. The immersion $\psi_\gamma$ also induces a complex structure on $T\gamma$ by the standard complex structure in $\c$. We call $T\gamma$ a complexified geodesic under this complex structure and denote it by $\mathfrak{C}_\gamma$, or just $\mathfrak{C}$ if the context is clear.

As $\gamma$ varies, the images of $\c - \r$ under the mapping $\psi_\gamma$ define a smooth foliation of $TM - 0_M$ by real surfaces. Each leaf extends across $M$ but the leaves will intersect each other on $M$. We called this foliation the \emph{Riemann foliation}. 
Let 
\begin{equation}
T^RM = \{ v\in TM : g(v,v) < R^2\}.
\end{equation} 
We define
\begin{defn}
An \emph{adapted complex structure} $J$ on $T^RM$ is a complex structure on $T^RM$ such that for any geodesic $\gamma$, $\mathfrak{C}_\gamma\cap T^RM$ is a complex submanifold of $T^RM$. 
\end{defn}

Adapted complex structures were introduced in \cite{ls,szoke} and, in an alternative form, at the same time in \cite{gs,gs2}. 

We need the following properties of adapted complex structures.

\begin{defn}
A \emph{parallel vector field} $\xi$ on a complexified geodesic $\mathfrak{C}$ is a vector field on $T^RM$ along $\mathfrak{C}$ such that $\xi$ is invariant under $N_\tau$ and $\phi_s$, where $\phi_s$ is the geodesic flow on $T^RM$. 
\end{defn}
Let $z\in T^RM - M$ and $\tilde{\xi}\in T_z T^RM$. Assume (WLOG) that $\psi_{\gamma}(i) = z$. Then there exists a parallel vector field $\xi$ along the leaf of the Riemann foliation containing $z$ such that $\xi(z) = \tilde{\xi}$. 
\vs
Let $\pi:TM\rightarrow M$ be the projection map and let
\[
 K:T(TM)\rightarrow \pi^*TM
 \] 
be the connection map. Parallel vector fields are completely determined by Jacobi fields along $\gamma$:
\begin{lemma} (\cite{ls})
$\xi|_\r$ is a Jacobi field $Y$ along $\gamma$, with $Y(\sigma) = \pi_*(\xi(\sigma + i))$ and $Y'(\sigma) = K(\xi (\sigma + i))$. Conversely,
if $Y$ is a Jacobi field along $\gamma$ with $Y(0) = u$ and $Y'(0) = v$, there exists a unique parallel vector field $\xi$ along $\mathfrak{C}_\gamma$ such that $\xi|_\r(\sigma) = Y(\sigma)$
\end{lemma}
Notice that if $\xi$ is tangent to the leaf at a point, then $\xi$ is tangent to the leaf at all points. 

Let $\gamma$ be an arc-length parametrized geodesic in $M$ and let $z=\dot\gamma(0)\in T_{\gamma(0)}M$. Choose an orthonormal basis $\{w_j\in T_{\gamma(0)}M\}_{1\leq j\leq n}$, such that $w_n = z$. Define parallel vector fields $\xi$ and $\eta$ such that 
\begin{equation}\label{pvf1}
\pi_* \xi_j (i) = w_j\quad K\xi_j (i) = 0;
\end{equation}
and
\begin{equation}\label{pvf2}
\pi_* \eta_j (i) = 0\quad K\eta_j (i) = w_j.
\end{equation}
Let $Y_j=\xi_j|_\r$ and $Z_j=\eta_j|_\r$ be Jacobi fields along $\gamma$. Notice that the $Y_j$ are pointwise linearly independent (except perhaps on a discrete subset $S$ of $\r$) Jacobi fields along $\gamma$. The $Z_j$'s are also smooth vector fields and hence there exist smooth functions $\tilde{a}_{jk}$ such that 
\begin{equation}
	Z_k = \sum \tilde{a}_{jk} Y_j
\end{equation}
on $\r - S$. The presence of an adapted complex structure ensures that \cite{ls} $\tilde{a}_{jk}$ has a (unique) meromorphic extension $a_{jk}$ over the strip
\[
\Sigma = \{ \sigma+i\tau\in\c, |\tau|<R/\sqrt{2E(z)} \}
\]
such that the poles of $a_{jk}$ lie on $\r$ and the matrix $\im(a_{jk})$ is symmetric and positive definite (hence invertible) in $\Sigma - \r$. Let $(e_{jk}) = (\im a_{jk})^{-1}$ and from \cite{ls}, for any point $p=\psi_\gamma (\sigma + i\tau)$, $0<\tau<R$, we have
\begin{equation}\label{j}
	\begin{array}{ll}
		&J_p \xi_h (\sigma + i\tau) \\  
			&	\\
		=&\sum e_{kh}(\sigma + i\tau) \left[\eta_k(\sigma + i\tau) - \left(\sum (\re\, a_{jk}(\sigma + i\tau))\xi_j(\sigma +i\tau)\right)\right]. \\
			&	
	\end{array}
\end{equation}
This shows that
\begin{prop}\label{uniqueacs}
Given a compact Riemannian manifold $(M,g)$ and $R$ such that $0<R\leq \infty$, there is at most one adapted complex structure on $T^RM$.
\end{prop}

Recall the canonical one form $\Theta$ on $TM$:
\begin{equation}
	\Theta (v) = g(z,\pi_* v)\quad v\in T_z (TM).
\end{equation}
Then $\Omega := d\Theta$ is the canonical symplectic form on $TM$. For $z\in TM$, we define
\begin{equation}
V_z = \ker (\Theta)_z \cap \ker (dE)_z \subset T_z (TM).
\end{equation}
By Theorem 5.3 in \cite{ls}, we know that $V_z$ is a $J$-invariant subspace. Let $z=\dot\gamma\in TM$. We also notice that $\xi$ is a parallel vector field such that $\xi(i)\in V_z$ if and only if $\xi|_\r$ is a normal Jacobi field along $\gamma$. 

We also notice that $\xi^{1,0}=\frac{1}{2}(\xi - iJ\xi)\in T^{1,0}X$ is a holomorphic section of $T^{1,0}(TM)$ along the geodesic. Regarding the adapted complex structures, we have the following relations.
\begin{lemma}\label{ddde} (\cite{ls}, Corollary 5.5)
$\dbar E - \d E = i\Theta$ and $\d\dbar E = \frac{i}{2}\Omega$.
\end{lemma}
\begin{lemma}\label{spshpsh} (\cite{ls}, Theorem 5.6)
$E$ is strictly plurisubharmonic and $u := \sqrt{E}$ is plurisubharmonic and satisfies the homogeneous complex Monge-Amp\`ere equation 
\[
	(\partial \bar{\partial}u)^n = 0
\]
on $T^RM - 0_M$.
\end{lemma}
\begin{lemma}\label{nantiholo} (\cite{ls}, Theorem 5.7)
$N_{-1}$ is an antiholomorphic involution of $T^RM$.
\end{lemma}

Since $N_{-1}$ is antiholomorphic, it follows that $M$, being the fixed point set of $N_{-1}$, is real analytic. This shows that analyticity is a necessary condition for the existence of the adapted complex structure. Looking at the existence of an adapted complex structure, we have
\begin{theorem}
(Theorem 2.2 \cite{szoke}) Let $M$ be a compact real analytic manifold equipped with a real analytic metric $g$. Then there exists $R>0$ such that $T^RM$ carries a unique adapted complex structure. 
\end{theorem}
\begin{rmk}
We will assume in what follows that $R$ is the maximal radius to which the adapted complex structure can be extended. We call $T^RM$ the \emph{Grauert tube} of $(M, g)$, and $R$ its radius. Since $T^RM$ admits a strictly plurisubharmonic exhaustion function, $T^RM$ is Stein.
\end{rmk}
\vs
\begin{rmk}
If there is an adapted complex structure on all of $TM$, {\em i.e.}, the maximal $R$ above is $\infty$, then we say that $M$ has an \emph{entire} tube. 
\end{rmk}
%
%
%
%
%
%
%%%%%%%%%%%%%%%%%%%%%%%%%%%%%%%%%%%%%%%%%%%%%%%%%%%%%%%%%%%%%%%%%%%%
%
%
%
%%%%%%%%%%%%%%%%%%%%%%%%%%%%%%%%%%%%%%%%%%%%%%%%%%%%%%%%%%%%%%%%%%%%
%								 Zoll spheres with entire Grauert tubes 
%%%%%%%%%%%%%%%%%%%%%%%%%%%%%%%%%%%%%%%%%%%%%%%%%%%%%%%%%%%%%%%%%%%%
%
%
%
%
% %%%%%%%%%%%%%%%%%%%%%%%%%%%%%%%%%%%%%%%%%%%%%%%%%%%%%%%%%%%%%%%%%%%
%
%
%
%
%
%
\section{Zoll spheres with entire Grauert tubes}\label{sec:entirezoll}
In this section we prove:
\begin{theorem}\label{thm:mainzollthm}
Let $M$ be a real analytic Riemannian manifold diffeomorphic to $S^n$ and Zoll with period $2\pi$, such that the adapted complex structure is entire $(R=\infty)$. Then $M$ is isometric to a round sphere. 
\end{theorem}
To prove this, we have to look at the properties of Zoll manifolds. \cite{besse}

Fix an arc-length parametrized geodesic $\gamma$ in $M$. The moduli space of oriented geodesics $N^+$ is a manifold and the tangent space of $N^+$ corresponds to the space of normal Jacobi fields (c.f. \cite{besse}). Explicitly, since $M$ is Zoll, the geodesic flow induces a free $S^1$ action on $S_1M$, the unit tangent bundle of $TM$ with respect to $g$, and we may identify $N^+ = S_1M/S^1$. A neighborhood of $\gamma\in N^+$ can be given as below. 

For $m\in\gamma$ let $N_\gamma M$ be the normal subspace to $\gamma$ at $m$ with respect to $g$ and let $U$ be a open neighborhood around $0$ in $N_\gamma M$. Let $u,v \in U \subset N_\gamma M$. Consider
\begin{equation}
\gamma_{u,v} (s) = \exp_{\exp_m u} s \left(P_u \frac{\dot{\gamma}(0) + v}{\sqrt{1+g(v,v)}}\right),
\end{equation}
where $P_u$ is the parallel transport along the curve $t\mapsto \exp_m tu$ to $t=1$. Then for $u,v\in U$, $\gamma_{u,v} (s)$ is a geodesic. We have 
\begin{equation}
P_u \frac{\dot{\gamma}(0) + v}{\sqrt{1+g(v,v)}}\in T_{\exp_m u}M
\end{equation}
with unit length. Let $\Pi$ be the projection map from the unit tangent bundle to $N^+$. Following \cite{besse}, we have
\begin{equation}
U\times U \ni (u,v)\mapsto \Pi \left( \exp_m u, P_u \frac{\dot{\gamma}(0) + v}{\sqrt{1+g(v,v)}}\right)
\end{equation}
is of maximal rank.

Relative to a local frame $\{e_i\}$ of $N_\gamma M$, let $u_i$ (resp. $v_i$) be the standard coordinates corresponding to $u$ (resp. $v$) relative to $e_i$, then $\frac{\d}{\d u_i}$ at $(u,v)=(0,0)$ corresponds to the Jacobi field $Y$ such that $Y(0) = e_i$ and $Y'(0) = 0$; $\frac{\d}{\d v_i}$ at $(0,0)$ corresponds to the Jacobi field $Z$ such that $Z(0) = 0$ and $Z'(0) = e_i$ \cite{besse}.

So $\gamma_{u,v} (s)$ is a $(2n-2)$-parameter family of geodesics around $\gamma$.

Now let $m = \gamma(0)$ and $z=\dot\gamma (0)\in T_m M$. Let $\{e_i\}_{i=1...n-1}$ be an orthonormal basis of $N_\gamma M$ at $m$. Define parallel vector fields $\xi_i$ and $\eta_i$ as in ~\eqref{pvf1} and ~\eqref{pvf2}. Then we have the meromorphic functions $a_{ij}$ on $\mathfrak{C}_\gamma$. 

As $\gamma$ is closed with length of $2\pi$, $\psi_\gamma$ is actually a map from $\c / 2\pi\z$ to $TM$. 

Now consider the map $\psi :\c / 2\pi\z \times U\times U\rightarrow TM$.
\begin{equation}
(\sigma + i\tau, u, v)\mapsto \left.\frac{d}{ds}\exp_{\exp_m u} s\tau \left(P_u \frac{\dot{\gamma}(0) + v}{\sqrt{1+g(v,v)}}\right)\right|_{s=\sigma}.
\end{equation}
Here $U$ is a small open neighborhood of $N_\gamma M$ around 0. Then $\psi$ restricted to $\{\tau >0\}\times U \times U$ (resp. $\{ \tau <0\}\times U \times U$) is a diffeomorphism onto its image when $U$ is small enough. Let $u = \sum u_j e_j$ and $v = \sum v_j e_j$. Using this diffeomorphism, we see that \cite{besse}
\begin{equation}
\xi_j(i)=\frac{\d}{\d u_j}\mbox{ at }(i,0,0);\quad\mbox{ and}
\end{equation}
\begin{equation}
\eta_j(i)=\frac{\d}{\d v_j}\mbox{ at }(i,0,0).
\end{equation}

Since $N_\tau$ and $\phi_\sigma$ commute with $\frac{\d}{\d u}$ and $\frac{\d}{\d v}$, the above two equations hold at any point in the leaf $\mathfrak{C}_\gamma$. This can also be interpreted as saying that $\frac{\d}{\d u_j}$ and $\frac{\d}{\d v_j}$ are given by Jacobi fields. As geodesic flows and reparametrizations preserve Jacobi fields (after reparametrizations), $\xi_j$ and $\eta_j$ correspond to $\frac{\d}{\d u_j}$ and $\frac{\d}{\d v_j}$, respectively, under this map.

Looking at the case of the CROSSes as cited in the introduction, $TS^n$ can be compactified to a smooth projective variety, namely, the non-singular quadric $\q^n \subset \c\p^{n+1}$. To compare our Zoll manifold $M$ with a round sphere, we first want to compactify $TM$ into a complex manifold $\bX$, and then compare $\bX$ with $\q^n$. 

Since $\psi_\gamma (\c-\r)$ gives a foliation of $TM-0_M$ by real surfaces, $TM-0_M$ is foliated by punctured disks $\D^*$ given by the positive and negative imaginary parts in $(\c-\r)/2\pi\z$. Thus we can compactify $TM$ to $\bX$ by filling in the punctures of these disks. Note that $\bX$ is a smooth manifold because the map $\psi:\D^*\times U\times U$ is a diffeomorphism at $\tau \neq 0$ when $U$ is small enough. Let $N^+:=\bX-TM$. Then $N^+$ corresponds to the set of oriented geodesics on $M$. Each complexified geodesic is compactified to a $\c\p^1\subset \bX$. We denote each compactified complexified geodesic by $C_\gamma$ or just $C$. 

Note that the parallel vector fields can be extended to smooth vector fields along the compactification since under $\psi$, they correspond to $\frac{\d}{\d u_j}$ and $\frac{\d}{\d v_j}$. Now we are ready to show that $\bX$ is indeed a complex manifold.
\begin{lemma}\label{extendj}
Let $J$ be the adapted complex structure on $TM$. Then $J$ can be extended to a complex structure on $\bX$.
\end{lemma}
\begin{proof}
If $J$ can be extended to $\bX$ smoothly, then $J$ is integrable on $\bX$ by continuity. So it suffices to extend $J$ smoothly to $\bar{X}$. 

Along $\mathfrak{C}$, the complex structure $J$ can be extended to $C$ by using the structure in $\p^1$. It suffices to extend the structure to the normal directions to $\mathcal{C}$. We can do this by looking at the action of $J$ on $\frac{\d}{\d u_j}$ and $\frac{\d}{\d v_j}$ locally in $\D\times U\times U$. We already know the action of $J$ on $\D^*\times U\times U$ as in (\ref{j}). To extend $J$ smoothly, we just have to extend $a_{ij}$ smoothly to the origin such that $\im (a_{ij})$ is invertible at the origin.

We know that $\im (a_{ij})$ is symmetric and positive definite in $\D^*$. Any diagonal entry in a positive definite matrix must be positive, which means $\im a_{ii} >0$ in $\D^*$. But $a_{ii}$ is holomorphic in $\D^*$ with $\im a_{ii} >0$. This shows that $a_{ii}$ can be extended to a holomorphic function in $\D$ by the Little Picard theorem or the fact that the upper half plane does not contain a neighborhood of $\infty$. Any principal minor of a symmetric positive definite matrix is positive. Let $i< j$, a $2\times 2$ principal minor
\begin{equation}
\begin{pmatrix}
\im a_{ii} &\im a_{ij}\\
\im a_{ij} &\im a_{jj}
\end{pmatrix}
\end{equation}
is positive definite with positive diagonal entries. This means 
\begin{equation}
(\im a_{ij})^2 < (\im a_{ii}) (\im a_{jj}).
\end{equation}
We already know $a_{ii}$ are holomorphic, hence bounded in a neighborhood of $0$. This means that in a neighborhood of 0, $\im a_{ij}$ is bounded. Using the same argument (Little Picard or $\infty$ neighborhood), we can see that $a_{ij}$ can be extended to a holomorphic function on $\D$. 

It suffices to show that $(\im a_{ij})$ is invertible at $0$. Fix $v\in \c^n-\{0\}$ and $A =(\im a_{ij})$ be the matrix. Then $v^t Av$ is a harmonic function on $\D$. Since $A$ is positive definite in $\D^*$ we have $v^t A v >0$ in $\D^*$. By the maximum principle of harmonic functions, we have $v^t A v >0$ in $\D$, which shows that $A$ is positive definite in $\D$, in particular $A$ is invertible. This shows that $J$ can be extended along each leaf of the compactified Riemann foliation.

To show that $J$ is smooth, we use the Cauchy integral formula. Note that for any $\gamma\in U$ near the origin, $\dd[,u_j]$ and $\dd[,v_j]$ correspond to smooth variations of parallel vector fields. Thus $a_{ij}^\gamma$ is smooth in $\D^*\times U \times U$ which can be extended to $\D\times U \times U$ smoothly along each $\D$. Using the Cauchy integral formula, $a_{ij}^\gamma (0) = \frac{1}{2\pi}\int_0^{2\pi} a_{ij}^\gamma(re^{i\theta}) d\theta$. But $a_{ij}^\gamma$ is smooth on $\D^*\times U\times U$. Thus $a_{ij}^\gamma$ is smooth on $\D\times U\times U$. Hence $J$ is smooth and this completes the proof.  
\end{proof}
By Lemma \ref{nantiholo}, we know that $N_{-1}$ is antiholomorphic in $TM$. We can extend this to an antiholomorphic map on $\bX$ by mapping $0$ to $\infty$ along each leaf. Again this is clearly antiholomorphic.

By construction, $J$ preserves $TN^+$. Thus $N^+$ is a complex submanifold of $\bX$.

Ultimately, we want to show that $\bX$ is the same as the model case, which is $\q^n$ and $N^+$ is $\q^{n-1}$. Notice that $\q^n$ is projective. So we would like to use the Kodaira embedding theorem. Let $\q^n\subset \c\p^{n+1}$. In an affine chart, let $\|z\|^2 = \sum_i z^i\bar{z}^i$. Then the K\"ahler potential of the Fubini-Study metric is $\log (1+\|z\|^2)$. On the other hand, in \cite{ls} and \cite{pw} we have $\cosh^{-1} \|z\|^2 = 2\sqrt{2E}$. This suggests that we should look at the potential function $\rho:=\log (1+\cosh 2\sqrt{2E})$ on $\bX$. In \cite{ls}, we know that $E=(\sqrt{E})^2$ is strictly plurisubharmonic and thus $\rho$ is strictly plurisubharmonic.
\begin{prop}\label{kahlerity}
The potential function $\rho :=\log (1+\cosh 2\sqrt{2E})$ defines a K\"ahler form on $TM$, which extends to a K\"ahler form in $\bX$.
\end{prop}
\begin{proof}
The only thing to prove here is to prove that the K\"ahler form given by $\rho$ extends to a K\"ahler form on $\bX$. To prove this, first we want to show that the K\"ahler form extends to $\bX$ continuously. To do this, look at each leaf in the Riemann foliation and let $z=\sigma + i\tau \in \c / 2\pi\z$. Here $\sigma$ is the unit speed geodesic parameter and $\tau$ is the length. Hence $\tau = \sqrt{2E}$. Let $\zeta$ be a local coordinate of each leaf near a point in $\bX - TM$. Then we have $\zeta = e^{iz}$ and $\zeta = 0$ corresponds to the point $p$ in $\bX - TM$ in each leaf. Let $U$ be a neighborhood of $p$. Then we have $\tau = -\log |\zeta|$. In this coordinate, in $U-\{ p \}$, 
\[
\rho = \log (1+\cosh 2\tau) = \log \dfrac{(1+|\zeta|^2)^2}{2|\zeta|^2}. 
\]
Notice that $\zeta$ is not holomorphic in $\bX$, but $\zeta$ restricted to each leaf is holomorphic. Let $w$ be a holomorphic coordinate in a neighborhood of $p$ in $\bX$ such that $w=0$ corresponds to $\bX-TM$ and $dw\neq 0$ at $w=0$. Then restricted to the leaf, both $\zeta$ and $w$ are holomorphic and vanish to first order at $p$. Thus we can write $\zeta = f\,w$ for some smooth function $f$, holomorphic along the leaf with $f(p)\neq 0$. It is easy to see that $f$ is locally a smooth function near $p$ in $\bX$ by the Cauchy integral formula, as above.

Returning to our function $\rho$, we have
\[
\rho = \log \frac{(1+|\zeta|^2)^2}{2|\zeta|^2} = \log \frac{(1+|fw|^2)^2}{2|fw|^2} = \log \frac{(1+|fw|^2)^2}{2|f|^2} - \log |w|^2.
\]
Since $w$ is a holomorphic coordinate in $\bX$ and $\d \dbar \log |w|^2 =0$, we see that the function $\log \frac{(1+|fw|^2)^2}{2|f|^2}$ defines the same K\"ahler form as $\rho$ does in $\bX-TM$. But this function is smooth near $p\in TM$ and thus it defines a form in $\bX$ near $p$. Repeating this process for different $p\in \bX - TM$ will give a form on $\bX$ that restricted to $TM$ is a K\"ahler form.

Now we have to show that this form is positive definite at $p\in \bX - TM$. To show this, we compute directly. Let $E=\frac{1}{2}g(x,v)$. By Lemma ~\ref{ddde} and Corollary 5.5 in ~\cite{ls}, we have 
\[
\d E \wedge \dbar E = \frac{i}{2} dE \wedge \Theta
\]
and 
\[
\d\dbar E = \dfrac{i}{2}\Omega.
\]
Thus we have 
\[
\d\dbar\sqrt{E} = \dfrac{i}{8E\sqrt{E}}(2E\,\Omega - dE\wedge \Theta)
\]
and
\[
\d \sqrt{E} \wedge \dbar \sqrt{E} = \frac{i}{8E} dE \wedge \Theta.
\]
Then the form is given by
\vs
\[
\begin{array}{lcl}
\omega:&=&-i\d\dbar \log (1+\cosh 2\sqrt{2E})\\
	&	&	\\
       &=&-i\d \left( \dfrac{2\sqrt{2}\sinh 2\sqrt{2E}}{1+\cosh 2\sqrt{2E}}\,\dbar\sqrt{E}\right)\\
	&	&	\\
	&=& -i\left(\dfrac{2\sqrt{2}\sinh 2\sqrt{2E}}{1+\cosh 2\sqrt{2E}}\,\d \dbar\sqrt{E} + \dfrac{8}{1+\cosh 2\sqrt{2E}}\d \sqrt{E}\wedge \dbar \sqrt{E}\right).\\
	&	&	\\
	&=&\dfrac{\sqrt{2}\sinh 2\sqrt{2E}}{4E\sqrt{E}(1+\cosh 2\sqrt{2E})}(2E\,\Omega - dE\wedge \Theta) \\
	&	&	\\
	&&+ \, \dfrac{1}{E(1+\cosh 2\sqrt{2E})}dE\wedge \Theta.
\end{array}
\]
\vs
Looking at the parallel vector fields $\xi_j$ and $\eta_j$ generated by $w_j$ again ($1\leq j\leq n-1$), we know that $\xi_j$ and $\eta_j$ restricted to $\r$ are normal Jacobi fields and thus they lie in $\ker \Theta \cap \ker dE$. We also know that $\xi_j$ and $\eta_j$ extend to $\bX$ smoothly, as independent vectors at $\bX-TM$. We also have
\[
	\Omega (\xi_j,\eta_j) = g(\pi_* \xi_j, K\eta_j) - g(\pi_*\eta_j, K\xi_j) = g(w_j,w_j) = 1
\] 
at $z=i$. Since $N_s^*\Omega = s\Omega$ for $s\in \r$, we have $\Omega(\xi_j,\eta_j) = \sqrt{2E}$ at any point in the leaf of $TM-0_M$.
Thus we have
\[
	\omega(\xi_j,\eta_j) = \frac{\sqrt{2}\sinh 2\sqrt{2E}}{2\sqrt{E}(1+\cosh 2\sqrt{2E})}\Omega (\xi_j,\eta_j) =\frac{\sinh 2\sqrt{2E}}{(1+\cosh 2\sqrt{2E})}.
\]
This means when $E\rightarrow\infty$ we have $\omega (\xi_j, \eta_j) \rightarrow 1\neq 0$. Since $\Omega (\xi_j, \xi_k) = \Omega(\eta_j,\eta_k) =0$ and $ \Omega(\xi_j,\eta_k)=0$ for $j\neq k$, by continuity $\omega (\xi_j,\xi_k)=\omega (\eta_j,\eta_k)=0$ and $\omega (\xi_j, \eta_k) =0$ for $j\neq k$. This applies to $\xi_n$ and $\eta_n$ in the sense that $\omega(\eta_n, \xi_j)=\omega(\eta_n, \eta_j) = 0$ and similarly for $\xi_n$ ($1\leq j\leq n-1$). We also have $\omega(\eta_n, \xi_n)\neq 0$ because this is the standard Fubini-Study metric on the $C$. Thus $\omega$ is non-degenerate in all directions at $p\in \bX-TM$. 
\end{proof}
Using the details of this proof, we can conclude that
\begin{lemma}\label{positivelinebundle}
$\o(N^+)$ is a positive line bundle.
\end{lemma}
\begin{proof}
We want to construct a hermitian metric $h$ on $\o(N^+)$ such that $-i\d\dbar\log h$ is the K\"ahler form above. As in the proof above, let $p\in N^+$ and $U_\alpha$ be a neighborhood of $p$ in $\bar{X}$. Using the notation above we set 
\[
h_\alpha (x)= \frac{(1+|f_\alpha(x) w_\alpha(x)|^2)^2}{2|f_\alpha(x)|^2}
\]
for $x\in U_\alpha$. Let $g_\alpha$ be a (holomorphic) defining function of $N^+$. Notice that we can pick $g_\alpha=w_\alpha$ since $w_\alpha$ is also a defining function on $U_\alpha$. This means, for another neighborhood $U_\beta$ of $p$, and $x\in U_\alpha \cap U_\beta -N^+$, we have
\[
\begin{array}{ll}
		& \\
	h_\alpha  &= \dfrac{(1+|f_\alpha w_\alpha|^2)^2}{2|f_\alpha|^2} \\
		& \\
		&= |w_\alpha|^2 \dfrac{(1+|f_\alpha w_\alpha|^2)^2}{2|f_\alpha w_\alpha|^2} \\
		& \\
		&= |w_\alpha|^2 \dfrac{(1+|\zeta|^2)^2}{2|\zeta|^2} \\
		& \\
		&= \dfrac{|w_\alpha|^2}{|w_\beta|^2}\dfrac{(1+|f_\beta w_\beta|^2)^2}{2|f_\beta|^2} \\
		& \\
		&= \dfrac{|w_\alpha|^2}{|w_\beta|^2} h_\beta . \\
		&
\end{array}
\]
Since all of $h_\alpha$, $h_\beta$ and $\frac{|w_\alpha|^2}{|w_\beta|^2}$ are continuous in $U_\alpha \cap U_\beta$, we have 
\[
h_\alpha = \frac{|w_\alpha|^2}{|w_\beta|^2} h_\beta 
\]
for all $x\in U_\alpha\cap U_\beta$.
Since $\frac{w_\alpha}{w_\beta}$ is the transition function of $\o(N^+)$, we successfully define a hermitian matrix such that $\o(N^+)$ is positive in a neighborhood of $N^+$. 

To complete the proof, we have to define $h$ away from $N^+$. This is easy, as we can just set 
\[
h=\frac{(1+|\zeta|^2)^2}{2|\zeta|^2} = \frac{1}{|w_\alpha|^2}h_\alpha
\]
Since 1 is a defining function of $N^+$ away from $N^+$, we have a hermitian metric such that $-i\d\dbar \log h = \omega$ is a K\"ahler form, hence positive. This completes the proof.
\end{proof}
Since $\bX$ is K\"ahler and has a positive bundle $\o(N^+)$, by Kodaira's embedding theorem, $\bX$ is projective and thus algebraic. In the following discussion we rename $N^+$ as $D$ to fit in the algebraic geometry world. $D$ is an ample divisor.  

Continuing our attempt to identify $\bX$ with $\q^n$, we next have to figure out what the cohomologies of $\bX$ are. 
\begin{prop}\label{audinaudin}
If $n$ is even, then
\[
H^i(\bX,\z)=\left\{
\begin{matrix}
0&\mbox{ if }&i\mbox{ is odd}\\
\z\oplus \z &\mbox{ if }&i=n\\
\z&\mbox{ if }&\mbox{otherwise.}
\end{matrix}
\right.
\]
If $n$ is odd, then 
\[
H^i(\bX,\z)=\left\{
\begin{matrix}
0&\mbox{ if }&i\mbox{ is odd}\\
\z&\mbox{ if }&i\mbox{ is even}
\end{matrix}
\right.
\]
\end{prop}
\begin{proof}
c.f. \cite{audin}. By the construction, the symplectic cut $W$ at $\{E = \frac{1}{2}\}$ of $TM$ in the Liouville symplectic form on $TM$ is diffeomorphic to $X$. Then $M\subset W$ is Lagrangian and the symplectic reduction at $E = \frac12$ is a symplectic (real) codimension-2 submanifold, with $E$ a Morse-Bott function with only $M$ and $D$ as critical sets. This means $W$ is a polarized symplectic manifold. By \cite{audin}, it has the above cohomology groups. Since $\bX$ is diffeomorphic to $W$, $\bX$ has the above cohomology groups.
\end{proof}
\begin{rmk}
In \cite{audin}, the above cohomology groups are generated by the Poincar\'e dual of $[D]^i$ over $\mathbb{Q}$, more precisely, over $\frac{1}{2}\z$, if they are isomorphic to $\z$. If $n$ is even, $H^n(\bX,\mathbb{Q})$ is generated by the Poincar\'e dual of $[D]^{n/2}$ and $[M]$ , with $[D]^n = 2$, $[D]\cdot[M] = 0$ and  $[M]^2 = -2$. (c.f. ~\cite{audin})
\end{rmk}
\begin{lemma}
Let $C\cong \c\p^1$ be a compactified complexified geodesic and $\mathcal{N}$ be the (holomorphic) normal bundle of $C$ in $\bX$. Then $\det \mathcal{N} \cong \o_{\c\p^1} (2n-2)$, where $n$ is the complex dimension of $\bX$.
\end{lemma}
\begin{proof}
Using previous notation, we have parallel vector fields  $\{\xi_j\}_{1\leq j\leq n-1}$ along $C = C_{\gamma}$. Then $\xi_j^{1,0}$ is a holomorphic section of $T^{1,0}(TM)|_{\mathfrak{C}_{\gamma}}$ (\cite{ls} Prop 5.1). By continuity in our extension, it is a holomorphic section of $T^{1,0}\bX|_{\mathcal{C}_{\gamma}}$. Since these parallel fields are not in $TC$, $\xi_j^{1,0}$ define holomorphic sections of $\mathcal{N}$. Now $\xi_j^{1,0}$ are linearly independent over $\c$ except at a discrete set of points in $M$, in which the $\xi_j^{1,0}$ vanish. This shows that $\bigwedge \xi_j^{1,0}$ defines a holomorphic section of $\det \mathcal{N} = \bigwedge^{n-1} \mathcal{N}$. $\bigwedge \xi_j^{1,0}$ has the vanishing order the sum of the vanishing order of $\xi_j^{1,0}$ because they are linearly independent over $\c$ except where one of them vanishes. But $\xi_j$ are given by Jacobi fields along $M$ and each Jacobi field vanishes at 2 points of multiplicity 1 (or 1 point with multiplicity 2). This shows that $\bigwedge \xi_j^{1,0}$ has vanishing order of $2n-2$ and therefore $\det \mathcal{N} = \o_{\c\p^1} (2n-2)$.
\end{proof}
\begin{lemma}
$K_{\bX} \cong \o(-nD)$
\end{lemma}
\begin{proof}
\emph{Case $n>2$}: By the adjunction formula we have $K_C = K_{\bX}|_C + \det \mathcal{N}$. By the above lemma we have $K_{\bX}|_C \cong \o(-2n)$. Since we know that $c_1(K_{\bX})$ defines a class in $H^2 (\bX; \z)$ and $h^2(\bX)=1$ and it is generated by $[D]$, we must have $K_{\bX} \cong \o(rD)$ for some $r\in \z$. Since $D$ intersects $C$ transversely at 2 points, we have $\o(D)|_C = \o(2)$. This shows that $\o(-2n) \cong K_{\bX}|_C \cong \o(rD)|_C = \o(2r)$ and so $r=-n$.

\emph{Case $n=2$}: We also have 
\[
\begin{array}{rcl}
c_1(\bX)|_D &=& c_1(TD\oplus N_{D|\bX})\\
 &=& c_1(D) + c_1(N_{D|\bX}) \\
 &=& c_1(\c\p^1) + [D]\cdot [D] \\
 &=& 2+2 = 4\in H^2(D,\z).
 \end{array}
\]
This shows that $c_1(\bX) \cdot [D] = 4$. Since $[D]$ and $[M]$ are generators of $H^2(\bX,\mathbb{Q})$, we can write $c_1(\bX) = \alpha [D] + \beta [M]$, with $\alpha,\beta \in \mathbb{Q}$. From $c_1(\bX) \cdot [D] = 4$, we see that $\alpha = 2$. By Proposition ~\ref{audinaudin}, the Euler characteristic of $\bX$ is 4. Thus by the Thom-Hirzebruch Signature Theorem, we have 
\[
c_1^2(\bX)=2\chi + 3\tau = 2(4) + 3(0) = 8.
\]
Thus we have $\beta = 0$. This shows that $c_1(\bX) = 2[D]$ or, equivalently, $K_{\bX} \cong \o(-2D)$.
\end{proof}
Using Kachi and Koll\'ar's argument ~\cite{kk} or Kobayashi-Ochiai ~\cite{ko}, $\bX$ is biholomorphic to the nonsingular quadric $\q^n$ by looking at the map $\bX\rightarrow \p H^0(\bX,\o_{\bX} (D))^*$.  
\begin{proof}[Proof of Theorem ~\ref{thm:mainzollthm}]
For $s\in H^0(\bX,\o_{\bX}(D))$, define $Ns:=\overline{N_{-1}^*s}$. Then $N$ is a conjugate linear involution from $H^0(\bX,\o_{\bX}(D))$ to itself. The set of fixed points of $N$ is a maximal totally real subspace $H^0(\bX,\o_{\bX}(D))_\r$ of $H^0(\bX,\o_{\bX}(D))$.   
We can choose $\tilde{s}_0,\ldots \tilde{s}_{n+1}\in H^0(\bX,\o_{\bX}(D))_\r$ which span $H^0(\bX,\o_{\bX}(D))$ over $\c$. We can let $s_0$ be the defining section of $D$ (since $D$ is $N$-invariant). Thus we have $f: x \mapsto [\tilde{s}_0(x):\ldots :\tilde{s}_{n+1}(x)]$ and $f$ commutes with the standard conjugation in $\c\p^{n+1}$. This shows that $N_{-1}$ extends to the standard conjugation $\tau : [z_i]\mapsto[\bar{z}_i]$ in $\c\p^{n+1}$.

In the $[z_i]$ coordinates, $f(\bX) = \q = \{ z^tAz = 0\} \subset \c\p^{n+1}$, where $A$ is a real symmetric matrix, up to a multiplicative constant. This means $A$ can be diagonalized by real symmetric matrix. Perform a real linear change of coordinates, we can assume $\q = \{ \pm z_1^2 \pm \ldots \pm z_{n+1}^2 = z_0^2 \}$. The fixed point set consists of all the real $z$'s. Since the fixed point of $\tau$ in $\q$ is an $n$-sphere, all the signs are positive and thus $\q = \{ \sum_{i=1}^{n+1} z_i^2 = z_0^2\}$. To conclude, there is a linear biholomorphism $G: \c\p^{n+1} \to \c\p^{n+1}$ such that $G(f(\bX)) = \q^n$ and such that it maps $M$, the fixed point of $N_{-1}$, to the fixed point of $\tau$, and the divisor $f(D)$ to $\q^{n-1}$.

Now both $u_1:=\sqrt{E}$ and $u_0$ are solutions to the HCMA equation on the affine quadric $\q_{aff} := \q^n - f(D)\subset \c^{n+1}$, where $u_0$ corresponds to the HCMA solution in the standard case of a CROSS \cite{pw}, i.e., the round metric of sectional curvature +1. We have $u_0 = \log |z|^2 + O(1)$ as $z\rightarrow\infty$ for the standard case, and by our construction, $u_1 = \log |z|^2 + O(1)$ as $z\rightarrow\infty$. Here $z$ is the coordinates of $\c^{n+1}$. Since $u_0$ and $u_1$ are both Siciak-Zaharjuta extremal functions of $S^n$ in $\q_{aff}^n$, $u_0 = u_1$. To see this directly, consider $u_1 - u_0$ restricted to a leaf of the foliation $\f_0$ associated to $u_0$ in $\q_{aff} - S^n$. It is subharmonic, and because it is $\equiv 0$ on $S^n$ and $O(1)$ at infinity, it is $\leq 0$ everywhere on that leaf. Since the leaves of $\f_0$ cover all of $\q_{aff} - S^n$, we see that $u_1 - u_0 \leq 0$ on $\q_{aff}$. But we can interchange the roles of $u_0$ and $u_1$, and conclude $u_0 \equiv u_1$. This means that the K\"ahler forms on the affine quadric given by $dd^cu_0^2, \, dd^cu_1^2$ are identical. Using the fact (\cite{ls}) that the metric $g_i$ on $M$ is the restriction of the K\"ahler metric on $\q_{aff}$ with K\"ahler form $\ddc u_i^2, i = 1,2,$ we see that $g_0$ and $g_1$ are identical on $M = S^n$.
\end{proof}
\begin{rmk}
Most of the procedures above can be applied to all the Zoll manifolds with entire tubes. Indeed, one can use the same argument to show that any Zoll manifold with an entire tube can be embedded into a projective complex manifold, with a smooth ample divisor at infinity which can be identified with the space of oriented geodesics on $M$. It is known that the only metrics on $\r\p^n$ which are Zoll are the CROSSes. It is conjectured, similarly, that the only Zoll metrics on the projective spaces over $\c$, $\h$ and $\mathbb{O}$ are the CROSSes \cite{tsukamoto}. 
\end{rmk}
%
%
%
%
%
%
%%%%%%%%%%%%%%%%%%%%%%%%%%%%%%%%%%%%%%%%%%%%%%%%%%%%%%%%%%%
%%%%%%%%%%%%%%%%%%%%%%%%%%%%%%%%%%%%%%%%%%%%%%%%%%%%%%%%%%%
%																				                 %	
%								       Algebraization: Sufficiency								         %	
%																				                 %	
%%%%%%%%%%%%%%%%%%%%%%%%%%%%%%%%%%%%%%%%%%%%%%%%%%%%%%%%%%%
%%%%%%%%%%%%%%%%%%%%%%%%%%%%%%%%%%%%%%%%%%%%%%%%%%%%%%%%%%%
%
%
%
%
%
%
\section{Algebraization: sufficiency}\label{sec:algsuff}
One could ask whether the extended complex normal Jacobi fields are crucial in the proof  ~\ref{extendj} that the complex structure extends to the points at infinity. We will show in this section that this is not so, and begin by localizing at infinity Stoll's notion \cite{stoll77} of a strictly parabolic exhaustion and recalling some of his notation.

Let $X$ be a complex Stein manifold of complex dimension $n$ with a strictly plurisubharmonic exhaustion function $\tau > 0$. Then we can define a K\"ahler metric $\Omega$, or ``$\tau$-metric'', with K\"ahler form equal to:
\[
\frac{i}{2}\d\dbar\tau = dd^c\tau > 0.
\]
In local coordinates, 
\[
\Omega = \sum_{i,j}\tau_{i\bj}dz^i \wedge d\bar{z}^j,
\]
where $\tau_i = \frac{\d\tau}{\d z^i}$ and $\tau_{\bi} = \frac{\d\tau}{\d \bar{z}^i}$. 

Define $\tau^{i\bj}$ by the relation
\[
\sum_j \tau^{i\bj}\tau_{k\bj} = \delta^i_k.
\]
We consider the complex gradient vector field of type $(1,0)$:
\begin{equation}{\label{eqn:Y}}
	Y:= \sum_{i,j}\tau^{i\bj}\tau_{\bj}\frac{\d}{\d z^i}.
\end{equation}
Let 
\begin{equation}
	\label{eqn:xieta}
		\xi = \re Y \and \eta = \im Y, 
\end{equation}
with flow $\psi$ and $\phi$ respectively. Let $u = \log \tau,$ and assume that there exists a $\tau_0 < +\infty$ such that whenever $\tau \geq \tau_0$,
\begin{enumerate}
\item  $u$ is plurisubharmonic and satisfies the HCMA equation (\ref{eqn:hcma}).
\vs
\item $\eta$ has a periodic flow $\phi$ and there is a free action of $S^1$ on $X_0:= X \cap \{ \tau\geq \tau_0\}$ corresponding to this flow. 
\end{enumerate}
Now since $\tau = e^u$, similar to the previous section, we have 
\[
\tau_{i\bj} = e^u(u_{i\bj} + u_iu_{\bj}).
\]
%Let $z = (z^i)$ such that 
%\begin{equation}\label{thisequation1}
%\sum_{i,j} u_{i\bj}z^i\bar{z}^j = 0.
%\end{equation}
%Then by the strict plurisubharmonicity of $\tau$, we must have $\sum_i u_iz^i \neq 0$. But there is at most 1 dimension of $z's$ satisfying condition ~\ref{thisequation1} and thus we have that the kernel of $u_{i\bj}$ is one dimension. 
Using \cite{bk}, there is a foliation of $X_0$ with leaves of one complex dimension and $dd^c u(v,-) = 0$ for all $v$ tangential to each leaf. We call this the Monge-Amp\`ere foliation. We denote a leaf of the foliation by $C$.

Notice that by Stoll \cite{stoll80}, $Y$ is tangent to the foliation in the sense that $Y\in T^{1,0}C$. This means that $\xi$ and $\eta$ are tangent to the leaves and they commute with each other. We also have that 
\[
Y\,\tau = \tau.
\]
Thus, we have
\begin{equation}
	\label{eqn:rtheta}
		\xi\,\tau = \tau, \mbox{ and } \eta\,\tau = 0.
\end{equation}
This means that $\eta$ preserves $\tau$ and thus $u$.
\begin{lemma}
Recall that $\psi$ and $\phi$ are the flows of $\xi$ and $\eta$ respectively, and recall that $\Omega = dd^c \tau >0$. We have $\psi_t^*\Omega = e^t\Omega$ and $\phi_t^*\Omega = \Omega$.
\end{lemma} 
\begin{proof}
By simple calculations, we have $\mathcal{L}_\xi \Omega = \Omega$ and $\mathcal{L}_\eta \Omega = 0$. Thus we have
\[
\frac{d}{dt} \phi_t^* \Omega = \phi_t^* \mathcal{L}_\eta \Omega =0,
\]
and
\[
\frac{d}{dt} \psi_t^* \Omega = \psi_t^* \mathcal{L}_\xi \Omega = \psi_t^*\Omega.
\]
Integrating, we get $\phi_t^*\Omega = \Omega$ and $\psi_t^*\Omega =e^t \Omega$. 
\end{proof}
Consider $w(t) = \tau(\psi_t(x))$. Then we have
\[
\begin{array}{rcl}
w(T) - w(0) &=& \tau(\psi_T(x)) - \tau(x) = \int_0^T \frac{d}{dt} \tau(\psi_t(x)) dt\\
&	&	\\
&=&\int_0^T d\tau_{\psi_t(x)} (\xi) dt\\ 
&	&	\\
&=&\int_0^T \xi_{\psi_t(x)} (\tau) dt\\
&	&	\\
&=&\int_0^T \tau (\psi_t(x)) dt\\
&	&	\\
&=&\int_0^T w(t) dt.\\
&	&	  
\end{array}
\]
Differentiating $w(T)$ with respect to $T$ gives a differential equation 
\[
w'(T) = w(T).
\]
Thus $w(T) = w(0)e^T$, which means 
\[
\psi_t^*\tau = \tau\circ\psi_t=e^t\tau.
\]
This means that $\psi_t^* u = t+u$.

Following Stoll's idea, we have the following:
\begin{lemma}
Let $C$ be a leaf of the foliation $F$. Then $C\cap \{ \tau > \tau_0\}\subset X_0$ is biholomorphic to a punctured holomorphic disk.
\end{lemma}
\begin{proof}
Let $x\in X_0$ be such that $\tau(x) = \tau_0$. Let $C$ be the leaf of the foliation containing $x$. Then $C$ is a one-dimensional complex manifold, and $\tau_0$ is a regular value of $\tau$ since $\xi\tau = \tau\neq 0$. Thus the set $\tau^{-1}(\tau_0)\subset C$ is a one real dimensional manifold. Let $y\in \tau^{-1}(\tau_0)\cap C$ and recall $\phi$ is periodic and non-trivial. Thus the connected components of $\tau^{-1}(\tau_0)\cap C$ are closed and thus they are diffeomorphic to $S^1$. If there are more than one connected components, say it contains more than one copy of $S^1$, by Morse theory, as it has no critical point in $C$, $C$ will be more than one copy of a cylinder. This is a contradiction because $C$ as a (maximal) leaf of a foliation, must be connected. This also shows that $C$ is a diffeomorphic to $S^1 \times \r$, which is a cylinder.
   
It is clear that the flow of $\phi$ is periodic of the same period for each punctured disk. This is because $\phi_s$ commutes with $\psi_t$: Assume $\phi$ has a period of $s_0$ at $x\in X_0$. Then 
\[
\psi_t(x) = \psi_t(\phi_{s_0}(x)) = \phi_{s_0}(\psi_t(x)),
\]
which shows that at $\psi_t(x)$, $\phi$ has the same period by continuity of $t$. Thus there is a diffeomorphism 
\[
\begin{array}{ccccc}
\r / s_0\z&\times& \r_{>0}&\rightarrow& C\\
(s&,&t)&\mapsto&\phi_s\psi_t (x).  
\end{array}
\]
Since $\psi$ and $\phi$ are real and imaginary components of a holomorphic vector field on $C$, by giving the set $\r / s_0\z\times \r_{>0}$ the standard complex structure, i.e. $J\frac{\d}{\d s} = \frac{\d}{\d t}$, the above map is a biholomorphism from $\r / s_0\z\times \r_{>0}$ to $C$. Let $u_0 = u(x)$. Then by composing with the exponential map 
\[
\zeta = e^{-\frac{2\pi}{s_0}(t+u_0 - is)},
\]
we have a map from a holomorphic disk to $C$, where 
\[
|\zeta| = e^{-\frac{2\pi}{s_0}(t+u_0)} = e^{-\frac{2\pi u}{s_0}}.
\]
\end{proof}
Then we can compactify each leaf:
\begin{corollary}
Each leaf can be compactified by filling the hole of each punctured disk. Thus $X$ can be compactified to a real $(2n)$-dimensional manifold $\bar{X}$.
\end{corollary}
\begin{proof}
We have that each leaf can be compactified by filling the hole of each punctured disk. Now in the neighboring leaf, $\zeta$ can be defined as a smooth coordinate around $C$ which restricts to holomorphic coordinates in each leaf. This can be done because the period $s_0$ is constant. Thus a neighborhood of $C$ is given by $\D^* \times \r^{2n-2}$. Thus $\bar{X}$ is a smooth manifold.
\end{proof}
Following Lempert and Sz\H{o}ke \cite{ls}, we want to describe the complex structure $J$ as we did in the previous section. But first, we define the notion of parallel vector fields. 

Pick a point $x\in X_0$ such that $\tau(x) = \tau_0$ and $C$ be the leaf containing $x$. Let $\tilde{v}\in T_x X_0$. Recall $\psi$ denotes the flow of $\xi$ and $\phi$ denotes the flow of $\eta$. Since $\phi$ and $\psi$ commute and $\phi$ is periodic along the leaf $C$, we have a well-defined vector field $v$ along $C$ such that it is invariant under $\phi$ and $\psi$. 

Let $V_x = \ker du \cap \ker d^c u$. Then $V_x$ is a $J$-invariant subspace and $T_xX_0 = V_x \oplus T_xC$. 
Let $\xi$ be a parallel vector field along $C$ such that $\xi_x \in V_x$ at a point $x$. Then by the relations $\psi_t^* u = t+u$ and $\phi_s^* u = u$, we have $\xi_y \in V_y$ for all $y\in C$. 
\begin{lemma}
Any parallel vector fields can be extended to the compactification $\bar{X}_0 = X_0 \cup \{ \bar{X} - X \}$.
\end{lemma}
\begin{proof}
Let $\gamma(r)$ be a curve in $X_0$ such that $\gamma(0)=x$ and $\gamma'(0) = \tilde{v}$. Extend $\tilde{v}$ to a parallel vector field $v$ along the leaf. Then $v$ at the corresponding point would be defined as $(\psi_t\phi_s\gamma)'(0)$. To extend the parallel vector field, one defines it to be $\displaystyle\lim_{t\rightarrow +\infty} (\psi_t\gamma)'(0)$. The limit exists in the compactification $\bar{X}_0$ and is well-defined because $\displaystyle\lim_{t\rightarrow+\infty} \psi_t =  \displaystyle\lim_{t\rightarrow+\infty} \psi_t\phi_s$ for any $s$.
\end{proof}
We would like to make use of some results in \cite{ls} now:
\begin{lemma}
$\Omega(v,w) = 0$ for $v\in V_x$ and $w\in T_x C$. Hence $\Omega = dd^c\tau$ restricts to a symplectic form on $V_x$.
\end{lemma}
\begin{proof}
Since $\tau = e^u$, we have 
\[
\dbar \d \tau = e^u (\d\dbar u + \dbar u \wedge \d u).
\]
Since $\d\dbar u = 0$ along $C$ and $V_x = \ker du\cap \ker d^c u$, we have $\Omega (v,w) = 0$.
\end{proof}
Since $V_x$ is $J$-invariant, symplectic with a symplectic form $\Omega$, and $\Omega$ is a $(1,1)$-form defined by a strictly plurisubharmonic function, we have $\Omega$ restricted to $V_x$ tamed and compatible with $J$. Thus there exists an orthonormal basis with respect to the metric $\Omega(\cdot,J\cdot)$
\[
\{\tilde\xi_1, J\tilde\xi_1,\ldots, \tilde\xi_{n-1}, J\tilde\xi_{n-1}\}
\]
such that $\Omega(\tilde\xi_i, \tilde\xi_j) = 0$ and $\Omega (\tilde\xi_i, J\tilde\xi_j) = \delta_{ij}$. Then $\{ \tilde\xi^{1,0}_i\}_{1\leq i\leq n-1}$ span $T^{1,0}V_x$. Letting $\tilde\xi_n = \xi$, we have $\{ \tilde\xi^{1,0}_i\}_{1\leq i\leq n-1}\cup \{ Y = \xi^{1,0}\}$ spans $T_x^{1,0} X_0$. 

Using Lempert-Sz\H{o}ke's construction, extend $\tilde\xi_i$ to parallel vector fields $\xi_i$ along $C$. Choose $n$ more vectors $\tilde\eta_i = J\tilde\xi_i$. Extend them to parallel vector fields $\eta_i$ along $C$. Notice that since $J$ is not invariant under $\phi_s$ and $\psi_t$, $\eta_i \neq J\xi_i$ in general. Thus following \cite{ls} again, we have
\begin{lemma}
The vector fields $\{\xi_j, J\xi_k\}$ are pointwise linearly independent over $\r$ on all of $C$. (Here $J\xi_k$ are not parallel.)
\end{lemma}
\begin{proof}
Since $\Omega(\tilde\xi_i, \tilde\xi_j)=0, i \neq j,$ and $\phi_s^*\Omega = \Omega$ and $\psi_t^*\Omega = e^t \Omega$, we have $\Omega(\xi_i, \xi_j) = 0, i \neq j,$ on all of $C$. Since $\Omega$ is of type $(1,1)$, we have $\Omega(J\xi_i, J\xi_j)=0$ on all of $C$. The rest follows from the proof of Proposition 6.4 of \cite{ls}.
\end{proof}
Since the $\{\xi_j,J\xi_k\}$ are pointwise linearly independent over $\r$ on all of $C$, we have $\{\xi_j^{1,0}\}$ pointwise linearly independent over $\c$ on $C$. Let
\[
\eta_k^{1,0} = \sum_j a_{jk}\xi^{1,0}_j\quad\mbox{ on } C,
\]
for a matrix of holomorphic functions $(a_{ij})$.  
Taking the real part, we have
\[
\eta_k = \sum_j (\re a_{jk}) \xi_j + \sum_j (\im a_{jk}) J\xi_j.
\]
Since $\{ \xi_j, \eta_k\}$ are pointwise linearly independent (because the $\psi_t$ and $\phi_s$ are diffeomorphisms onto their images), $(\im a_{jk})$ is invertible.
Let $(e_{jk}) = (\im a_{jk})^{-1}$. Then we have
\[
J\xi_h = \sum_k e_{kh} \left( \eta_k - \sum_j (\re a_{jk}) \xi_j\right).
\]
If we can show that the holomorphic functions $a_{jk}$ can be extended to the compactifications near $\bar{X}_0 - X_0$ and their imaginary parts are invertible, we can extend $J$ to $\bar{X}_0$ and hence $\bar{X}$ as an almost complex structure. (Here $J$ will be smooth, by the Cauchy integral formula, similar to the proof of Lemma ~\ref{extendj} in the above section.) Since $D := \bar{X} - X$ is of measure zero, by continuity, the Nijenhuis tensor will also vanish at $D$ and therefore $J$ will be integrable everywhere. Lastly, $D$ will be $J$ invariant by the above formula, which in turn implies that $D$ will be a complex submanifold of $\bX$.

To show that $a_{jk}$ can be extended, we will show
\begin{lemma}
$(e_{ij})$ is a symmetric matrix on all of $C$. Thus $(\im a_{ij})$ is symmetric on all of $C$.
\end{lemma}
\begin{proof}
We have
\[
\Omega(\xi_i, J\xi_h)=\Omega(\xi_i, e_{kh} (\eta_k - (\re a_{jk}) \xi_j))=\Omega(\xi_i, e_{kh}\eta_k)=e_{kh}\Omega(\xi_i, \eta_k).
\]
At $x'\in X_0$ such that $x' = \psi_t \phi_s(x)$, $\Omega_{x'} = e^t \Omega_x$. Thus at $x'$, we have
\[
\Omega_{x'}(\xi_i, J\xi_h) = e_{kh}\Omega_{x'}(\xi_i, \eta_k) = e_{kh}e^t \Omega_x (\xi_i, \eta_k) = e_{kh}e^t \delta_{ik} = e_{ih} e^t.
\]
But $\Omega$ is of type (1,1) and therefore 
\[
\Omega (\xi_i, J\xi_h) = \Omega(\xi_h, J\xi_i),
\]
Thus we have $e_{ih} = e_{hi}$ and therefore $\im a_{jk} = \im a_{kj}$.
\end{proof}
\begin{corollary}
$(\im a_{jk})$ is symmetric and positive definite on all of $C$. 
\end{corollary}
\begin{proof}
Since $(\im a_{jk}) = Id$ at $x$, and the matrix is invertible on all of $C$, it is positive definite on all of $C$. 
\end{proof}
\begin{corollary}
$J$ can be extended to a complex structure on $\bar{X}_0$.
\end{corollary}
\begin{proof}
The proof is completely the same as Lemma \ref{extendj}.
\end{proof}
Now we claim that
\begin{prop}
The potential function $\rho := \log (1 + \cosh \frac{4\pi u}{s_0})$ defines a K\"ahler form on $X_0$, which extends to a K\"ahler form on $\bar{X}$.
\end{prop}
\begin{proof}
In a neighborhood of $D$, we have $\tau$ is strictly plurisubharmonic, and $u=\log \tau$ is plurisubharmonic and it satisfies the HCMA equation. Using the strict plurisubharmonicity of $\tau$, after some simple calculations, $u^2$ is strictly plurisubharmonic near $D$ and so is $\frac{16\pi^2}{s_0^2}u^2$. Thus by Lemma 3.1 from Patrizio-Wong \cite{pw}, we have that $\rho$ is strictly plurisubharmonic near $D$. Since $X$ is a Stein manifold, we can extend $\rho$ to a strictly plurisubharmonic function on all of $X$. To do this, first we pick a smooth function $\rho_1$ on $X-X_0$ such that it agrees with $\rho$ on the boundary $\{ \tau = \tau_0\}$. The set $X-X_0$ is holomorphically convex, thus there exists a strictly plurisubharmonic function $\rho_2$ on $X-X_0$ such that it vanishes at the boundary. Since $\overline{X-X_0}$ is compact, there exists $C_1$ such that $\rho_1 + C_1\rho_2$ is strictly plurisubharmonic on $X-X_0$ such that it agrees on $\rho$ on the boundary. Thus $\rho$ can be extended to a function which is strictly plurisubharmonic on $X$ except on $\{ \tau = \tau_0\}$. Smoothing $\rho$ near the boundary gives a strictly plurisubharmonic function on $X$ such that it agrees with $\rho$ when $\tau$ is large enough. Thus it defines a K\"ahler metric on $X$.

Notice that we have 
\[
\rho = \log (1 + \cosh \frac{4\pi u}{s_0}) = \log (1+\cosh (-\log |\zeta|^2)),
\]
and from the proof of Proposition \ref{kahlerity}, 
\[
\omega := -i\d\dbar\rho
\]
can be extended to a form on $\bar{X}$.

To prove the form is positive definite, the method is exactly the same as in the proof of Proposition \ref{kahlerity}:
\[
\begin{array}{rcl}
\omega:&=&-i\d\dbar \log \left(1+ \cosh \dfrac{4\pi u}{s_0}\right)\\
&	&	\\
&=& -i\left(\dfrac{8\pi i}{s_0 e^u}\dfrac{\sinh \frac{4\pi u}{s_0}}{1+\cosh \frac{4\pi u}{s_0}}\Omega 
+ \dfrac{4\pi}{s_0}\dfrac{\frac{4\pi}{s_0} - \sinh \frac{4\pi u}{s_0}}{1+\cosh \frac{4\pi u}{s_0}}\d u \wedge \dbar u\right).\\
&	&	
\end{array}
\]
For parallel vector fields $\xi_j$ and $\eta_j$ lying in $V_y$ for all $y\in C$, we have
\[
\Omega(\xi_j, \eta_j) = 1
\]
at $x$. Since $\psi_t^*\Omega = e^t \Omega$, and $u = u_0 + t$, we have
\[
\begin{array}{rcl}
\omega(\xi_j, \eta_j) &=& \dfrac{8\pi}{s_0 e^u}\dfrac{\sinh \frac{4\pi u}{s_0}}{1+\cosh \frac{4\pi u}{s_0}}\Omega(\xi_j, \eta_j)\\
	&	&	\\
&=&\dfrac{8\pi}{s_0 \tau_0}\dfrac{\sinh \frac{4\pi u}{s_0}}{1+\cosh \frac{4\pi u}{s_0}}\\
	&	&	\\
&\rightarrow & \dfrac{8\pi}{s_0\tau_0}\neq 0.\\
	&	&
\end{array}
\]
as $u\rightarrow +\infty$. Similarly we have $\Omega(\xi_j, \xi_k) = \Omega(\eta_j, \eta_k) = 0$ and the same holds for $\xi_n$ and $\eta_n$, as in the proof of Proposition ~\ref{kahlerity}. Thus $\omega$ defines a K\"ahler form on $\bar{X}$.
\end{proof}
Similar to Proposition \ref{positivelinebundle} we can prove that $\o(D)$ is a positive line bundle. Thus by Kodaira's embedding theorem, $\bar{X}$ is projective algebraic, and $D$ is ample.

This is a special case of an algebraization question raised by the first author, as mentioned in the introduction:
\begin{ques}
	\label{aq}
Let $X$ be a Stein manifold, $\tau$ a smooth strictly plurisubharmonic exhaustion function on $X$. Assume that there exists $\tau_0>0$ such that on $X_0=\{ \tau\geq \tau_0\}$, $u:=\log \tau$ is plurisubharmonic and satisfies the HCMA equation $(\d\dbar u)^n = 0$. Then is $X$ an affine algebraic variety with coordinate ring $R := f \in \o(X)$ such that
\begin{equation}
	\label{coordring}
		|f|\leq C (1+\tau)^N,\, \text{for constants} \;C = C_f, N = N_f <+\infty?
\end{equation}
\end{ques}
We have proved that 
\begin{theorem}
	\label{thm:algn}
The question \ref{aq} has a positive response if we assume that the action given by the flow of $\phi$ on $X_0$ is freely periodic.
\end{theorem}
\begin{proof}
The only thing we have to prove here is that the coordinate ring is given by ~\eqref{coordring}. Any algebraic function $f$ on $X$ can be extended to a meromorphic function $g$ on $\bar{X}$. Since $\tau = \frac{1}{|\zeta|^2} = \frac{1}{|fw|^2}$, where $w$ is a holomorphic coordinate near a point on $V$ and $f$ does not vanish at $V$, c.f. the proof of ~\ref{kahlerity}. Thus $g$ is locally bounded by a polynomial of $\tau$. But $\bar{X}$ is compact. Thus $g$ is bounded by a polynomial of $\tau$. 

Conversely, if $g$ is bounded by a polynomial of $\tau$, $g$ is bounded by a polynomial of $w$, a holomoprhic coordinate system near a point on $V$. Thus $g$ is meromorphic and thus it is a rational function \cite{gaga} which is holomorphic when restricted to $X$. Thus $g$ is algebraic on $X$.
\end{proof}
%
%
%
%
%
%
%%%%%%%%%%%%%%%%%%%%%%%%%%%%%%%%%%%%%%%%%%%%%%%%%%%%%%%%%%%%%%%%%%%%
%%%%%%%%%%%%%%%%%%%%%%%%%%%%%%%%%%%%%%%%%%%%%%%%%%%%%%%%%%%%%%%%%%%%
%																													      %	
%												Algebraization: necessity			  						   			      %	
%																									    				      %			
%%%%%%%%%%%%%%%%%%%%%%%%%%%%%%%%%%%%%%%%%%%%%%%%%%%%%%%%%%%%%%%%%%%%
%%%%%%%%%%%%%%%%%%%%%%%%%%%%%%%%%%%%%%%%%%%%%%%%%%%%%%%%%%%%%%%%%%%%
%
%
%
%
%
%
\section{Algebraization: necessity}\label{sec:algnec}
One could also ask the converse of theorem \ref{thm:algn}: if $X$ is a smooth affine variety, does there exist a plurisubharmonic exhaustion function $u$ satisfying the HCMA equation for $u\geq u_0$, such that $\tau := e^u$ is strictly plurisubharmonic outside a compact set and determines the entire functions on $X$ which are algebraic (rational)? Unfortunately, we do not know the answer to this, in general. Note that in Theorem \ref{thm:algn} we can, starting from a p.s.h. solution $u$ of the HCMA equation for which $\tau = e^u$ is strictly p.s.h., find a smooth, projective compactification $\bar{X}$ of $X$ such that the divisor $D := \bar{X} - X$ is smooth and ample. In this restricted situation we can find an exact converse to Theorem \ref{thm:algn}. The statement we prove is local at the divisor $D$.
\begin{theorem}
	\label{thm:converse}
Let $\bar{X}$ be a complex manifold of dimension $n$, and $D \subset \bar{X}$ a compact, smooth divisor with holomorphic normal line bundle $L$ on $D$ which is ample. Then there is a smoothly bounded tubular neighborhood $\v$ of $D$ and a foliation $\f$ of $\v$ by complex disks crossing $D$ transversally, and a solution $u$ on $\v - D$ to the HCMA equation on $\v - D$ for which $u$ is smooth up to the boundary $\partial{\v}$ of $\v$, with $u|_{\partial{\v}} = 0$, and such that for $z_0 \in D$ and in the local parameter $\zeta$ uniformizing the leaf $\triangle_{z_0}$ of $\f$ passing through $z_0$, with $\zeta(z_0) = 0$, we have
			\begin{equation}
				\label{eqn:logatinf}
					u|_{\triangle_{z_0} - \{z_0\}} = \log \frac1{|\zeta|^2}.
			\end{equation}
\end{theorem}
Theorem \ref{thm:converse}, especially the condition (\ref{eqn:logatinf}), shows that polynomial growth with respect to $\tau$ determines the algebraic functions on $X$.
\begin{corollary}
	\label{algebraic}
		If $X$ in theorem \ref{thm:converse} is affine algebraic, the rational regular functions on $X$ are the entire functions $f \in \o (X)$ which verify 
			\begin{equation}
				\label{eqn:growth}
					|f| \leq C_N (1 + \tau)^N,
			\end{equation}
		for some $N >> 0$ in $\z$, and some $C_N \in \r^+$ which may depend on $f$.
\end{corollary} 

Here $\tau > 0$ is a strictly plurisubharmonic extension to all of $X$ of the function $e^{u}$, defined near $D \subset \bar{X}$. It is easy to generalize Theorem \ref{thm:converse} to the situation where $X$ has isolated singularities, or is a smooth resolution of such a space.
\vs
The proof of theorem \ref{thm:converse} follows the usual pattern for finding classical solutions to the HCMA equation, i.e., one first finds the foliation $\f$ and then defines the solution in terms of the formula (\ref{eqn:logatinf}). We find the foliation here in two steps, first deforming the complex manifold $\u$ to another $\u_0$ where there is a model solution of HCMA as desired (see section \ref{ss:geomsetup} below), with foliation $\f_0$, etc., and then using the continuity method to deform this model foliation and solution to the original case $\u$, as desired. In \cite{lempert81} this is done for the Kobayashi disks through a point $z_0 \in D$, a strictly convex domain in $\c^n$, by finding mappings critical for the Kobayashi functional at $z_0$ (the infinitesimal variational problem defining the Kobayashi disks). Here we fix attention at $z_0 \in D,$ our divisor at infinity, and consider a functional which is dual in a vague sense to the Kobayashi functional. It extremizes either the Robin constants associated to disks through $z_0$, or minimizes the distortion of maps through $z_0$ spanning $\u$. For details see subsection \ref{ss:vars} below. 
\vs
Our methods follow those of \cite{lempert81} closely, though our interest being purely local at $D$, we only need to make a small perturbation of the model solution $u_0$ on $\u_0$, rather than needing a global continuity argument. Technically, this means we can use an easy implicit function theorem argument on a Banach manifold, where the evaluation of the differential at the reference point is equivalent to evaluating the second variation of the distortion functional just mentioned. We do not have a global condition such as strict real convexity as in \cite{lempert81}, however, so we must compensate with local calculations which restrict the degree to which we can deform our model solution, equivalently, restrict the size of the neighborhood of $D \subset \bar{X}$ on which we can find a solution of the HCMA equation as desired in theorem \ref{thm:converse}. It seems an interesting open question to ask whether one can say something globally about the size of such special neighborhoods of $D$. Some special cases of this question are known, e.g., \cite{lempert85}.
\vs
%%%%%%%%%%%%%%%%%%%%%%%%%%%%%%%%%%%%%%%%%%%%%%%%%%%%%%%%%%%%%%%%%%%%
%
%
%%%%%%%%%%%%%%%%%%%%%%%%%%%%%%%%%%%%%%%%%%%%%%%%%%%%%%%%%%%%%%%%%%%%
 								% Geometric Set-up %
%%%%%%%%%%%%%%%%%%%%%%%%%%%%%%%%%%%%%%%%%%%%%%%%%%%%%%%%%%%%%%%%%%%%
%
%
%%%%%%%%%%%%%%%%%%%%%%%%%%%%%%%%%%%%%%%%%%%%%%%%%%%%%%%%%%%%%%%%%%%% 
\subsection{Geometric set-up}
	\label{ss:geomsetup}
	
	Let $L$ be the total space of the normal bundle of $D\subset \bar{X}$, with $\pi: L \to D$ the projection, and $h$ a hermitian metric on $L$ with positive $c_1(L,h)$. In terms of a local holomorphic frame $\sigma \neq 0$ on $U \subset D$, we have $L|_U \cong U \times \c$ with coordinates $z, t$, there is a positive representative $h = h(z) > 0$ such that the length-squared function $\|\ell \|^2$ for $\ell = t \sigma$ in $L$ is given by $h(z) |t|^2$. The model solution of the HCMA equation on $L$ is just 
	\[
			u_0 := \log \frac1{\|\ell\|^2} = \log \frac1{h |t|^2}, \; \text{locally}.
	\]
	
	It is easy to verify that $u_0$ possesses the properties claimed in Theorem \ref{thm:converse}, if we take $\u_0 =$ the sublevel set $\{\ell \in L \,| \, \|\ell\|^2 \leq 1\}.$ In particular, $u_0$ is p.s.h. because 
	\[
			\ddb u_0 = \ddb \pi^* \frac1{h} = \pi^* c_1(L, h) \geq 0 \, \text{on $L - \{0\}$},
	\]
by assumption. That $\tau_0 := e^{u_0}$ is strictly p.s.h. on $L- \{0\}$ follows from the fact that on $L_{z_0} =\pi^{-1}(z_0)$ we have in local coordinates $\tau = \frac1{h(z_0) |t|^2}$. Similarly, the solution is strictly periodic because then $\eta_0$ (\ref{pvf1}) is tangent along every $L_{z_0}$, and equal to $\dd[\;,\theta]$ there, if $\zeta = \rho e^{i\theta}$ in polar coordinates. Thus the flow of $\eta_0$ is periodic of exact period $2\pi$ there.
\vs
To construct a deformation of $\u_0$ and the solution $u_0$ there, consider the holomorphic line bundle $L(D)$ on $\v$ (or $\bar{X}$), which has a holomorphic section $s = s_D$ which vanishes exactly to first order exactly along $D$. $L(D) |_D \cong L$, and we let $h$ denote now an extension of the metric $h$ originally on $L$ to all of $L(D)$. Without loss of generality, extending the definition of $u_0 := \log \frac1{\|s\|^2}$ on all of $L(D)$, we can assume that 
	\begin{equation}
		\label{eqn:modelpole}
		\ddb u_0 = \ddb \log \frac1{\|s\|^2} \geq 0
	\end{equation}
on $\v - D$: we may shrink $\v$ and replace $h$ by $e^{\mu \|s\|^2} h$, for $\mu \in \r, \mu >> 0$, if necessary. Define $B_r(L(D)) := \{ \|\ell\|^2 \leq r^2\} \subset L(D),$ for $r > 0.$ Define a set $\u \subset L(D) \times \c$ by
\[
\tilde{\u} = \{ (\ell, \lambda) \in L(D) \times \c \; | \; \lambda \ell = s_D(\pi_{\v}(\ell))\},
\]
where $\pi_{\v}: L(D) \to \v$ is the projection, and
\[
\u = B_1(L(D)) \cap \tilde{\u}.
\]
Note that this is consistent with the use of $\u_0$ above. $\u$ has a projection $p$ to $\c$, and we set $\u_{\lambda} = p^{-1}(\lambda) \subset L(D)$.
Let $s$ be the defining section of $D$ and set 
\[
\v_{\lambda} = \pi (\u_{\lambda}) \subset \v.
\]
for $\lambda \neq 0$. Note that $\pi (\u_0) = D,$ and that 
\[
\v_{\lambda} = \{\|s\| \leq |\lambda|^2\}.
\]
\vs
Fixing a $z_0 \in D$, we can find a coordinate patch $U \subset \bar{X}$ and local coordinates $z = (z_1,\ldots, z_n)$ on $U$ such that $z(z_0) = 0.$
We can assume there is a non-vanishing holomorphic section $\sigma$ of $L(D)$ over $U$ with fiber coordinate $t$ such that on $U$, $s = s_D = z_n \sigma.$ Hence, over $U$, $s(U)$ is defined by the equation $t = z_n.$ It follows that, over $U \times \c$, $\tilde{\u}$ is defined by 
\[
\tilde{\u} = \{ F := \lambda t - z_n = 0\}, 
\]
and 
\[
B_1(L(D)) \times \c = \{(z,t, \lambda) \, | \, h(z) |t|^2 < 1\}.
\]	
Since $\dd[F, z_n] \equiv -1,$ we see that $\tilde{\u}$ is a smooth submanifold of $L(D) \times \c$, and that we may use $(z_1,\ldots, z_{n-1}, t, \lambda)$ as coordinates on $\tilde{\u}$. We note also that
\[
d\lambda \wedge dF \equiv - d\lambda \wedge dz_n \; \text{mod terms not involving} \; d\lambda \wedge dz_n,
\]
and so $p: \tilde{\u} \to \c$ is a fiber bundle with smooth fibers $\tilde{\u}_{\lambda}$. Finally, since $B_1(L(D)) = \{ \log \|\ell \|^2 \leq 0\} = \{2\log \rho + \log h \leq 0\}$, where $t = \rho e^{i\theta}$, and
\[
d\log \|\ell \|^2 \wedge d\lambda \wedge dF \equiv - 2 \, \frac{d\rho}{\rho} \wedge d\lambda \wedge dz_n,
\]
modulo terms not involving $d\rho \wedge d\lambda \wedge dz_n$, we have that $p: \u \to \c$ is a fiber bundle with fibers $\u_{\lambda}$ which are smooth manifolds with boundary. Similarly, $\v_{\lambda}, \lambda \neq 0,$ are smooth manifolds with boundary in $\v$, for $|\lambda|$ small enough.
\vs
\vs
For every $\lambda \in \c$, the curvature property of $L(D), h$ guarantees that the complex manifold with boundary $\u_{\lambda}$ has a smooth, strongly pseudoconcave boundary in $\tilde{\u}_{\lambda}$, for any $\lambda \in \c$.
\vs
We may further arrange our local coordinates, replacing the local frame $\sigma$ of $L(D)$ by $\tilde{\sigma} = g \cdot \sigma,$ where $g = g(z)$ is holomorphic and non-vanishing on a neighborhood of $U$, so that in this new frame,
\[
h(0) = 1\and \dd[h,z_i](0) = \frac{\partial^2 h}{\partial z_i \partial z_j}(0) = 0, \forall i, j, k,
\]
if we choose the two-jet of $g$ appropriately at $z_0$, i.e., $z = 0$. This implies
\[
\frac{\partial^3 h}{\partial z_i \partial z_j \partial \bar{z}_k}(0) = 0, \forall i, j, k.
\]
We can further change the coordinates $z$ complex linearly at $z= 0$, so that 
\[
\frac{\partial^2 h}{\partial z_j \partial \bar{z}_k}(0) = \delta_{j,\bar{k}}, \forall j, k.
\]
\vs
\vs
\vs
\noindent {\em Remarks:}	
\vskip 1mm
\noindent 1. In what follows, $\lambda$ will be our deformation parameter for perturbation of the model solution and its foliation. 
\vs
\noindent 2. The above is a standard construction which, in the case where $X$ affine algebraic compactifies to a smooth, projective $\bar{X} \subset \p^N$, with $D := \bar{X} \pitchfork H_{\infty}$ is the transverse intersection of $\bar{X}$ with the hyperplane at infinity $H_{\infty}$, is just the deformation of $\bar{X}$ to the cone over the divisor at infinity $D$. This is the case when $D$ is very ample. In order to obtain an exact converse to theorem \ref{thm:algn} we have to allow $D$ to be just ample, equivalently, to intersect the hyperplane at infinity with tangential multiplicity $\geq 1$.

%%%%%%%%%%%%%%%%%%%%%%%%%%%%%%%%%%%%%%%%%%%%%%%%%%%%%%%%%%%%%%%%%%%%
%
%
%%%%%%%%%%%%%%%%%%%%%%%%%%%%%%%%%%%%%%%%%%%%%%%%%%%%%%%%%%%%%%%%%%%%
 								% Variational Problems %
%%%%%%%%%%%%%%%%%%%%%%%%%%%%%%%%%%%%%%%%%%%%%%%%%%%%%%%%%%%%%%%%%%%%
%
% 
%%%%%%%%%%%%%%%%%%%%%%%%%%%%%%%%%%%%%%%%%%%%%%%%%%%%%%%%%%%%%%%%%%%%
\vs
\vs
\subsection{Variational problems}
	\label{ss:vars}
	Consider a point $z_0 \in D$ and let $u_0$ be as in (\ref{eqn:modelpole}). Let $\zeta$ be the standard coordinate on the unit disk $\D$, and let $f: (\D, 0) \to (\u, z_0)$ be a holomorphic map such that $f'(0) = \dd[f,\zeta](0) \neq 0$ and is transverse to $D$ at $z_0$. Define
	\begin{equation}
		\label{eqn:shrobin}
			v_f(\zeta) = f^*u_0 - \log \frac1{|\zeta|^2}.
	\end{equation}
It is easy to check that, by the transversality condition on $f, v_f$ is subharmonic on $\D^*$ and extends smoothly across $0 \in \D$. Define the Robin constant $R(f) := v_f(0).$ This definition depends on the choice of a reference p.s.h. function $u_0$ with the correct logarithmic pole order along $D$, but this is by an additive constant which is independent of $f$. 

With such a choice in hand, let us first fix the parameters $z_0 \in D, \lambda \in \D(\rho), \rho << 1, \, k \, \text{an integer}\, \geq 2, \and \alpha \in (0,1),$ and consider the space of competitors
	\[
		\HH^{k,\alpha}_{\lambda, z_0}(\D, \u_{\lambda}) := f:(\D, 0) \to (\u_{\lambda}, z_0) 
	\]
	\begin{equation}
		\label{eqn:competitors}
	 \text{such that}\;
		\left\{{\begin{array}{cl}
			(1) & f \, \text{is of H\"older class}\, \mathcal{C}^{k,\alpha} \,\text{on}\, \bar{\D} \\
				&	\\
			(2) & f \,  \text{is a closed embedding of}\, \bar{\D} \subset \bar{\u_{\lambda}} \\
				&	\\
			(3) & f(\bar{\D}) \pitchfork D = \{z_0\}.
		\end{array}}\right.
	\end{equation}
	
For each such $f \in \comp$, we have $v_f|_{\partial \D} \equiv 0$, and so 
\begin{equation}
	\label{eqn:Robbd}
		R(f) \leq 0.
\end{equation}
Set 
	\[
		R(\u_{\lambda}, z_0) := \sup_{f \in \comp} R(f).
	\]
\vs
\vs
Suppose we have a p.s.h. solution $u = u_{\lambda}$ of the HCMA equation as in Theorem \ref{thm:converse}, and we use $u_{\lambda}$ in place of $u_0$ in the definition (\ref{eqn:shrobin}) of $v_f$ and hence of $R(f)$, we will change the value of $R(\u_{\lambda}, z_0)$, but not the set of competitors which would maximize $R(f)$. Then let $f_{z_0}$ be the parametrization of the disk leaf of the foliation $\f_u$ through $z_0$ by $\D$, such that $f(0) = z_0$. Then 
\[
	f_{z_0}^* u - \log \frac1{|\zeta|^2} \equiv 0 \on \D,
\]
whereas for any other competitor $f \in \comp,$ we have $f^*u|_{\partial \D} \equiv 0,$ and so $v_f \leq 0 \on \D$ and $R(f) \leq 0.$ Therefore, $0 = R(\u_{\lambda}, z_0),$ and if any other competitor $f$ has $R(f) = 0$, then $f^*u = \log \frac1{|\zeta|^2} \on \D^*$. By the rank condition on the complex Hessian of $u$, $f(\D)$ must lie within a leaf of the foliation $\f_u$, and then necessarily the leaf passing through $z_0$. Schwarz's lemma then says that $f(\zeta) = f_{z_0}(e^{i\theta_0} \cdot \zeta)$ for a suitable choice of $\theta_0 \in \r/2\pi \z.$	Thus, to construct the foliation $\f$ for the solution $u$ sought in Theorem \ref{thm:converse} we should construct it by foliating $\u_{\lambda}$ by extremal disks for the Robin functional at the various points $z_0 \in D$. It is not clear at this point that such a foliation exists, although it clearly does for $\u_0$ on $\u_0$.
\vs

Using the coordinates $(z', t, \lambda) = (z_1, \ldots, z_{n-1}, t, \lambda)$ as in the previous section, which cover a neighborhood of $[\u_0 \cap \pi^{-1}(z_0)] \times \D$. For the model solution, we see that $f_0(\zeta, z', \lambda = 0) = (z', \frac{e^{i\theta_0}}{h(z')^{\frac12}}\cdot \zeta, 0)$ are the minimizing embeddings, for $\theta_0 \in \r$. We see explicitly that, for 
$f \in \compvar[z',\lambda]$, 
\begin{equation}
	\label{eqn:r&d}
		R(f) = \log \frac1{|\dd[f_n, \zeta](0)|^2} + c,
\end{equation}
where $f(\zeta) = (f_1(\zeta,\ldots,f_{n-1}(\zeta), t = f_n(\zeta), \lambda) \in \u_{\lambda}$ and $f(0) = (z', t = 0, \lambda).$ Here $c$ is a real constant which depends on $z', \lambda$ (as well as the choice of the reference function used for comparisons), but not on $f$. Hence, to find the $R(\u_{\lambda}, z')$, it suffices to find $\inf_{f \in \compvar[z',\lambda]_0} \dd[f_n,\zeta](0),$ where
\[
\compvar[z',\lambda]_0 = \{ f \in \compvar[z',\lambda] \, | \, \dd[f_n, \zeta] \in \r\}.
\]
\vs
\vs
\noindent Thus, ``dually" to the Kobayashi problem, we seek to {\em minimize} the distortion $\dd[f_n,\zeta](0) \in  \r^+$ over embedded disks which {\em span} \, $\u_{\lambda}$ and pass tranversally once across $D$ at $z'$. Note that there is a strictly positive lower bound for $\dd[f_n,\zeta](0)$, locally uniformly in $z', \lambda$, and $f \in \compvar[z',\lambda]_0,$ by (\ref{eqn:r&d}) and the fact that $R(f)$ is uniformly bounded from above.
\vs
To find critical points for the functional $\dd[f_n,\zeta](0)$ in each $\compvar[z',\lambda]_0$, we will use a standard implicit function theorem on a Banach manifold of embedded disks. First, we will define the moduli space of disks $\m_{loc} := $\, the set of all $f: (\D, 0) \to (\u, D)$ which satisfy
\begin{enumerate}
	\item $f$ is holomorphic and of class $\mathcal{C}^{k, \alpha}(\bar{\D})$,
	\vs
	\item $f$ is an embedding of $\bar{\D}$ into $\bar{\u}$ such that $f(\partial \D) \subset \partial \u$,
	\vs
	\item $f(0) = z' \in D, \and f(\D) \cap D = f(\D) \cap \{t = 0\} = \{z'\}$, 
	\vs
	\item $\dd[f_n,\zeta](0) > 0$,
	\vs
	\item $p\circ f \equiv \lambda$, for some $\lambda \in \c$.
\end{enumerate}
These are the same as the corresponding items in (1) - (3) in the definition (\ref{eqn:competitors}) of $\comp$ above, except that $f(\D) \cap D = \{z'\}$ and $\lambda$ are now variable. Condition (4) is the restriction for guaranteeing the uniqueness of extremals. In fact, there is an evaluation map $e: \m_{loc} \ni f \to (f(0), p(f(0))) \in D \times \c$, and $e^{-1}(z',\lambda) = \compvar[z',\lambda] = \m_{loc, z', \lambda}$. It will follow from the arguments below that the map $e$ is a smooth fiber bundle of Banach manifolds.
\vs
Let us abbreviate by $r = r(z', t, \lambda) = \|\ell\|^2 = h(z', z_n = \lambda t)\cdot |t|^2,$ expressed in our special coordinates. To show that $\m_{loc}$ is a Banach manifold, we have to figure out what its tangent spaces should be, and then graph it over its tangent space smoothly, locally. So suppose we have a curve $f(a) = f(\zeta, a) = (f_1(\zeta, a), \ldots, f_{n-1}(\zeta, a), f_n(\zeta, a), \lambda(a)) \in \m_{loc}$, smooth in the parameter $a \in (-\epsilon, \epsilon)$, and therefore satisfying 
\begin{equation}
	\label{eqn:bc}
		r(f(\zeta, a)) \equiv 1, \forall \zeta \in \partial \D.
\end{equation}	
For simplicity, we will assume that $f = f(\zeta, a)$ has constant $z' = (f_1(0),\ldots, f_{n-1}(0))$ and constant $\lambda$. It will be clearly seen below that these are free parameters, and we will just keep track of the smooth dependence of our constructions on these parameters.
\vs
Differentiating (\ref{eqn:bc}) at $a = 0$, we get 
\begin{equation}
	\label{eqn:tgt}
		2 \re (r_t(f(\zeta, 0)) \, \cdot \delta f_n(\zeta)) + 2 \re (\sum_{i = 1}^{n-1} r_i(f(\zeta, 0)) \cdot \delta f_i(\zeta)) = 0, 
\end{equation}
for all $\zeta \in \partial \D.$ Note that the functions $r_i = \dd[r, z_i](f(\zeta, 0)),$ and $r_t = \dd[r, t](f(\zeta, 0))$ are in $\mathcal{C}^{k,\alpha}$ on $\partial \D$.
\vs
Consider the extremal map $f_0(\zeta) = (0, \zeta, 0) \in \compvar[z_0, 0]_0$. Then 
\begin{equation}
	\label{eqn:qf0}
		r_t(f_0(\zeta)) = \bar{\zeta},\;  r_i(f_0(\zeta)) \equiv 0, \, \forall \zeta \in \partial \D.
\end{equation}
Now we will assume, which is sufficient for our purposes, that our $f \in \m_{loc}$ is sufficiently close to $f_0$ in $\mathcal{C}^{k,\alpha}(\bar{\D}, \bar{\u})$ that $r_t(f(\zeta)) \neq 0 \in \c,$ and has winding number -1 on $\partial \D$. Hence, $\zeta r_t(f)$ has winding number 0, and by the classical Riemann-Hilbert problem, we can write 
\begin{equation}
	\label{eqn:RHrt}
		\zeta \, r_t(f(\zeta)) = \rho(\zeta) g(\zeta), \zeta \in \partial \D,
\end{equation}
where $\rho(\zeta) = \rho(\zeta, f) > 0$ on $\partial \D$, and $g$ extends to a nowhere vanishing holomorphic function on $\D$, also denoted $g$, such that $|g(0)| = 1$. Such $\rho, g$ are unique, in $\mathcal{C}^{k,\alpha}(\partial \D)$, and are smooth functions of $r_t(f)$, and hence, of $f$. Since we want $\delta f_n(0) = 0$, we look for $w = \zeta^{-1} \delta f_n,$ and from (\ref{eqn:tgt}) we have
\[
\re [g \cdot w] = - \frac1{\rho} \, \re(\sum_{i = 1}^{n-1} r_i(f(\zeta, 0)) \cdot \delta f_i(\zeta)) ,
\]
from whence we conclude $g \cdot w = q,$ where
\begin{equation}
	\label{eqn:q}
q = \\- \frac1{\rho} \, \re(\sum_{i = 1}^{n-1} r_i(f(\zeta)) \cdot \delta f_i(\zeta)) + i T(\frac1{\rho} \, \re(\sum_{i = 1}^{n-1} r_i(f(\zeta)) \cdot \delta f_i(\zeta))) + i b,
\end{equation}
where $b = b(f)$ is a real constant determined so that $\im w(0) = 0.$ 
\vs
In (\ref{eqn:q}), $T$ is the classical Hilbert transform, so that if $u = u(\zeta), \zeta \in \partial \D$, and $U(\zeta)$ its Poisson integral, harmonic on $\D$, then $T(u)$ is the boundary value of $U^*(\zeta)$, where $U + \sqrt{-1} U^*$ is holomorphic on $\D$ with $U^*(0) = 0$. Then it is well-known that
\[
T: \mathcal{C}^{k,\alpha}(\partial \D, \r) \to \mathcal{C}^{k,\alpha}(\partial \D, \r)
\]
continuously, and if $\mathcal{C}^{k,\alpha}_0(\partial \D, \r) = u \in \mathcal{C}^{k,\alpha}(\partial \D, \r)$ with
\[
\int_0^{2\pi} u \; d\theta = 0,
\]
then
\[
T: \mathcal{C}^{k,\alpha}_0(\partial \D, \r) \to \mathcal{C}^{k,\alpha}_0(\partial D, \r)
\]
and is a topological isomorphism. We extend $T$ complex linearly to complex valued functions in $\mathcal{C}^{k,\alpha}_0(\partial \D, \c)$. Then 
\[
\HH_0^{k,\alpha} \ni f \stackrel{T}{\to} \; -i \cdot f.
\]
Finally, for $f \equiv A + i B, A, B \in \r,$
\[
T(f) = A.
\]
\vs
Returning to  (\ref{eqn:q}), if $f = f_0$, then $g(f_0) \equiv 1$ ({\em cf.} (\ref{eqn:qf0})), and so, by the continuity of $g = g(f)$ as a function of $f$, we have, for $f$ sufficiently close to $f_0$ in $\mathcal{C}^{k,\alpha}$-norm, that we may write $g(0)$ uniquely as $g(0) = e^{i \theta_0}, \theta_0 \in (-\frac{\pi}2, \frac{\pi}2).$ Now we want
\begin{equation}
	\label{eqn:impart}
		\im w(0) = \im (e^{-i\theta_0} \cdot (a + i b)) = 0, 
\end{equation}
where
\[
a = - \frac1{2\pi} \int_0^{2\pi} \frac1{\rho} \, \re(\sum_{i = 1}^{n-1} r_i(f(\zeta, 0)) \cdot \delta f_i(\zeta)) \; d\theta.
\]
Solving (\ref{eqn:impart}) gives
\begin{equation}
	\label{eqn:b}
		b = \tan (\theta_0) \cdot a
\end{equation}
which is a real linear function of $(\delta f_1, \ldots, \delta f_{n-1})$. Hence we have arrived at a uniquely determined
\begin{equation}
	\label{eqn:deltan}
		\delta f_n = \zeta \cdot g(f)^{-1} \cdot q \in \HH^{k,\alpha}_0(\D),
\end{equation}
with $q$ as in (\ref{eqn:q}) and $b$ is fixed by (\ref{eqn:b}). Define $T_f(\m_{loc, z', \lambda}) $ as the set
\begin{equation}
	\label{eqn:defTf}
	\{\delta f = (\delta f_1,\ldots, \delta f_n) \, \in \mathcal{B} \times \HH^{k, \alpha}_0 \, | \, \delta f_n \, \text{satisfies} \, (\ref{eqn:deltan})\}, 
\end{equation}
where the real Banach space $\mathcal{B}$ is $\HH^{k,\alpha}_0(\D)^{\oplus (n-1)}$. We will justify the notation below.
\vs
Next, still for $z', \lambda$ fixed, we would like to show that $\m_{loc, z', \lambda}$ is a manifold modeled on the real Banach space $\mathcal{B}$ near $f_0$. To do  this, let us seek a small $w \in \HH^{k,\alpha}$ with $\im w(0) = 0,$ for $\delta f = (\delta f_1,\ldots, \delta f_{n-1}, 0) \in T_{f_0}(\m_{loc, 0, 0})$ sufficiently close to $0 \in T_{f_0}(\m_{loc, 0, 0})$, so that 
\begin{equation}
	\label{eqn:bdrycond}
		r(z_1 + \delta f_1,\ldots, z_{n-1} + \delta f_{n-1}, \zeta + \zeta \cdot w, \lambda) \equiv 1,
\end{equation}
for all $\zeta \in \partial \D$.
\vs
Writing $w = u + i v,$ with $u, v$ $\r$-valued, we can differentiate (\ref{eqn:bdrycond}) with respect to $u$ and evaluate at $w = 0$:
\[
\dd[\;,u][r(\delta f_1 + z_1,\ldots, \delta f_{n-1} + z_{n-1}, \zeta + \zeta \cdot w, \lambda)] |_{w=0} = 2 \re \zeta r_t(f).  
\]
When $z' = 0, \lambda = 0$, we have $\zeta r_t(f_0) \equiv 1$, so for all $\delta f$ close to $0 \in T_{f_0}(\m_{loc}, 0, 0)$, $w$ close to $0 \in \HH^{k,\alpha}$ such that $\im w(0) = 0,$ and $z', \lambda$ close to 0, we have 
\[
\dd[\;,u][r(\delta f_1 + z_1,\ldots, \delta f_{n-1} + z_{n-1}, \zeta + \zeta \cdot w, \lambda)] |_{w=0} > 0.
\]
Hence, we can solve equation (\ref{eqn:bdrycond}) for $u$ in terms of $v$:
\begin{equation}
	\label{eqn:u}
		u = G(v) = G(f, \delta f, z', \lambda, v),
\end{equation}
where $G$ is smooth in all of its arguments. Given $v$ small enough in $\mathcal{C}^{k,\alpha}_0$, we define 
\begin{equation}
	\label{eqn:u0}
		u(0) = \frac1{2\pi} \int_0^{2\pi} G(f, \delta f, z',\lambda, v) \, d\theta,
\end{equation}
and $\tilde{u} = u - u(0) \in \mathcal{C}^{k,\alpha}_0$. If $w$ is holomorphic, then $Tu = v$, and $Tv = - \tilde{u}$, we can apply $T$ to both sides of (\ref{eqn:u}) to get 
\begin{equation}
	\label{eqn:v}
		v = T(G(f, \delta f, z', \lambda, v)) := F(v),
\end{equation}
an equation for $v$ of fixed point type on $\mathcal{C}^{k,\alpha}_0.$ Substituting (\ref{eqn:u}) into (\ref{eqn:bdrycond}), we get
\[
r(z' + \delta f, \zeta \cdot (1 + G(v) + i v), \lambda) \equiv 1.
\]
Differentiating this with respect to $v$, we get
\begin{equation}
	\label{eqn:dv}
		2 \re \, (\zeta r_t \cdot (G_v + i )) = 0.
\end{equation}
When $\delta f =0, \and v =0,$ we get $\zeta r_t(f_0) \equiv 1$, and hence $G_v \equiv 0$ on $\partial \D$, which in turn implies that the derivative of $DF(f_0, 0) = DF(f_0, v = 0, z' = 0, \lambda = 0): \mathcal{C}^{k,\alpha}_0 \to \mathcal{C}^{k,\alpha}_0$ is the 0 operator, and $v =0$ is a contracting fixed point for $F$. Hence for $\delta f \in T_{f_0}(\m_{loc, 0, 0})$ close to 0, $v$ close to $0 \in \mathcal{C}^{k,\alpha}_0$, and $z', \lambda$ close to 0, we have the operator norm $\|DF\| \leq \epsilon.$ If we choose $\epsilon$ small enough, then we conclude by the contraction mapping theorem that there is a small neighborhood of $\delta f = 0 \in T_{f_0}(\m_{loc, 0, 0})$ and a unique solution $v \in \mathcal{C}^{k, \alpha}_0$ to (\ref{eqn:v}), and this solution depends smoothly on $\delta f, z', \and \lambda$. Finally, for $\delta f \in T_{f_0}(\m_{loc, 0, 0})$, the map
\begin{equation}
	\label{eqn:param}
		\delta f \to (\delta f_1 + z_1, \ldots, \delta f_{n-1} +z_{n-1}, \zeta(1 + w(\zeta, \delta f, z', \lambda)))
\end{equation}
is a smooth parametrization of $\m_{loc, z', \lambda}$ near $f_0$. It is now easy to verify that the tangent space to this Banach manifold at a given $f$ is the space $T_f(\m_{loc, z', \lambda})$ as given above in (\ref{eqn:defTf}).
\vs
Now we can return to our functional 
\begin{equation}
	\label{eqn:defE}
		\begin{array}{rcl} \m_{loc, z', \lambda} \ni f \to E(f) & := & \dd[f_n,\zeta](0) \\
			&	&	\\
			& = 	& \frac1{2\pi} \re \int_0^{2\pi} \zeta^{-1} f_n(\zeta) \; d\theta.
		\end{array}
\end{equation}
We can calculate the $L^2$-gradient of $E$ at $f$ directly now, and letting $\gamma(a)$ be a smooth curve in $\m_{loc, z', \lambda}$ with $\dot{\gamma}(0) = \delta f \in T_f(\m_{loc, z', \lambda})$, we get
\begin{equation}
	\label{eqn:GradE}
		\begin{array}{rcl} \frac{d \;\;}{da} E(\gamma(a))\, |_{a=0} & = & \frac1{\pi} \, \int_0^{2\pi} \re[ \zeta^{-1} \delta f_n ]\; d\theta \\ & & \\ & = & \frac1{\pi} \, \int_0^{2\pi} \re[ g(f)^{-1} \cdot q ]\; d\theta,
		\end{array}
\end{equation}
using (\ref{eqn:deltan}), (\ref{eqn:RHrt}) and (\ref{eqn:q}). But
\[
		\frac1{\pi} \, \int_0^{2\pi} \re[ g(f)^{-1} \cdot q ]\; d\theta = 2 \re \, \frac{q(0)}{g(0)},
\]
where $g(0) = e^{i\theta_0},$ and
\[
q(0) = - (1 + i \, \tan \theta_0) \; \frac1{2\pi} \, \int_0^{2\pi}  \re \, [\sum \frac{r_i}{\rho} \delta f_i] \; d\theta,
\]
so that
\[
\begin{array}[c]{rcl}
	&&\\
	2 \re \, \frac{q(0)}{g(0)} & = & - 2 \re [\frac{1 + i \tan \theta_0}{\cos \theta_0 - i \sin \theta_0}] \cdot \frac1{2\pi} \, \int_0^{2\pi}  \re \, [\sum \frac{r_i}{\rho} \delta f_i] \; d\theta\\
		&	&	\\
		& = & - \frac{\sec \theta_0}{\pi} \; \int_0^{2\pi}  \re \, [\sum \frac{r_i}{\rho} \delta f_i] \; d\theta.\\
		&&
\end{array}
\]
Since each $\delta f_i \in \HH^{k,\alpha}_0$, we can rewrite this last expression as 
\[
- \frac{\sec \theta_0}{\pi} \, \int_0^{2\pi} \, \re (\sum \, S_0(\bar{\frac{r_i}{\rho}}) \, \overline{\delta f_i}) \; d\theta,
\]
where 
\[
S_0: \mathcal({C}^{k,\alpha})^{\oplus (n-1)} \to (\HH^{k,\alpha}_0)^{\oplus (n-1)} = \mathcal{B}
\]
is the Szeg\H{o} projector. Hence, we find
\begin{equation}
	\label{eqn:gradEfinal}
		\nabla E(f) = - 2 \sec \theta_0 \; (S_0( \,\overline{\frac{r_1}{\rho}} \, ), \ldots, S_0( \, \overline{\frac{r_{n-1}}{\rho}} \, )) \in \mathcal{B}.
\end{equation}
Note that we take the $L^2$ inner product on $(\mathcal{C}^{k,\alpha})^{\oplus (n-1)}$ to be 
\[
(h_1,h_2) = \frac1{2\pi} \, \int_0^{2\pi} \re[ \sum h_{1, \ell}, \overline{h_{2,\ell}}] \; d\theta.
\]
In particular, we have an important corollary to (\ref{eqn:gradEfinal}):
\begin{corollary}
	\label{cor:rihol}
		$E$ is critical on $\m_{loc, z', \lambda}$ at $f$ if and only if each $\frac{r_{\ell}}{\rho}$ on $\partial \D$ extends holomorphically to $\bar{\D}, \ell = 1, \ldots, n-1$.
\end{corollary}
\begin{proof}
Indeed, $S_0(\frac{r_{\bar{\ell}}}{\rho}) = 0$ if and only if $\frac{r_{\ell}}{\rho}$ is the boundary value of a holomorphic function.
\end{proof}
\vs
\vs
%%%%%%%%%%%%%%%%%%%%%%%%%%%%%%%%%%%%%%%%%%%%%%%%%%%%%%%%%%%%%%%%%%%%
%
%
%%%%%%%%%%%%%%%%%%%%%%%%%%%%%%%%%%%%%%%%%%%%%%%%%%%%%%%%%%%%%%%%%%%%
 								% Perturbation argument %
%%%%%%%%%%%%%%%%%%%%%%%%%%%%%%%%%%%%%%%%%%%%%%%%%%%%%%%%%%%%%%%%%%%%
%
%
% %%%%%%%%%%%%%%%%%%%%%%%%%%%%%%%%%%%%%%%%%%%%%%%%%%%%%%%%%%%%%%%%%%%	
\subsection{Perturbation argument}
	\label{ss:pert}
We are now in a position to find critical points of $E$ on $\m_{loc, z', \lambda}$ for $(z', \lambda) \neq (0,0)$. Consider now the smooth map of Banach manifolds
\[
\nabla E: \m_{loc} \to \mathcal{B}.
\]
where $\nabla E (f)$ as in (\ref{eqn:gradEfinal}), is the gradient computed only along the fiber $\m_{loc, f(0), \lambda(f)}$. Note that $\m_{loc, 0, 0} \subset \m_{loc}$ and so $T_{f_0}(\m_{loc,0,0}) = \mathcal{B} \subset T_{f_0}(\m_{loc})$. 

%We will compute $D(\nabla E)(f_0)\, |_{T_{f_0}(\m_{loc,0,0})} \in \,\text{Lin}(\mathcal{B}, \mathcal{B})$, and then use the implicit function theorem to construct critical points.
%
\vs
\begin{lemma}
	\label{lem:svf0}
\[
D(\nabla E)(f_0)\, |_{T_{f_0}(\m_{loc,0,0})} = 2\, I \in \,\text{Lin}(\mathcal{B}, \mathcal{B}).
\]
\end{lemma}
\vs
Given the lemma, the implicit function theorem implies the desired corollary:
\vs
\begin{corollary}
	\label{cor:H}
There is a unique smooth function $H = H(z', \lambda)$ from a neighborhood of $0 \in \c^{n-1} \times \c$ into $\m_{loc}$ verifying the following:
\begin{enumerate}
	\item $H(0,0) = f_0.$
	\vs
	\item $e(H(z', \lambda)) = (z', \lambda) \in \c^{n-1} \times \c.$
	\vs 
	\item $\nabla E(H(z', \lambda)) \equiv 0 \in \mathcal{B}.$
\end{enumerate}
Furthermore, $H(z', \lambda)$ is an isolated critical point in $\m_{loc,z',\lambda}$.
\end{corollary}
\begin{proof}(of Lemma \ref{lem:svf0})
We calculate directly, component by component from (\ref{eqn:gradEfinal}), so the $\ell$-th component of $\nabla E(f(a))$ is 
\[
[\nabla E(f(a))]_{\ell} = -2 \sec (\theta_0(f(a)) S_0(\frac{r_{\bar{\ell}}}{\rho(f(a))}), \, \ell = 1, \ldots, n-1. 
\]
Note first that we have 
\[
	\frac{d\;\;}{da} \sec \theta_0(f(a)) |_{a =0} = 0,
\] 
since $\theta_0(f_0) = 0$. Therefore
\[
\frac{d\;\;}{da} [\nabla E(f(a)]_{\ell}\,(f_0) = -2 S_0(\frac{d\;\;}{da} (\frac{r_{\bar{\ell}}}{\rho})(f_0)).
\]
But then
\begin{equation}
	\label{eqn:Hess}
\begin{array}{rcl}
\frac{d\;\;}{da} (\frac{r_{\bar{\ell}}}{\rho})(f_0) & = &  \sum_k \, \frac{r_{k,\bar{\ell}}(f_0)}{\rho(f_0)} \cdot \delta f_k(f_0) + \frac{r_{\bar{k},\bar{\ell}}(f_0)}{\rho(f_0)} \cdot \overline{\delta f_k}(f_0) \\
							&										& \\
							& & - \sum_k \, \frac{r_{\bar{\ell}}(f_0) \rho_k(f_0)}{\rho^2(f_0)} \cdot \delta f_k(f_0) + \frac{r_{\bar{\ell}}(f_0) \rho_{\bar{k}}(f_0)}{\rho^2(f_0)} \cdot \overline{\delta f_k}(f_0).
\end{array}
\end{equation}
We now use the following normalizations. (1) is from (\ref{eqn:qf0}) and (\ref{eqn:RHrt})  above. (2) - (4) are translations of those at the end of section \ref{ss:geomsetup} to the present $r$: 
\begin{enumerate}
\item $\rho(f_0) \equiv 1.$
\vs
\item $r_{\ell}(f_0) \equiv 0, \for \ell = 1, \ldots, n-1.$
\vs 
\item $r_{k,\ell}(f_0) \equiv 0, \for k, \, \ell = 1,\ldots, n-1.$
\vs
\item $r_{k,\bar{\ell}}(f_0) \equiv - \delta_{k,\bar{\ell}}, \for k, \, \ell = 1,\ldots, n-1.$
\end{enumerate}
\vs
Using these, (\ref{eqn:Hess}) reduces to
\[
\frac{d\;\;}{da} (\frac{r_{\bar{\ell}}}{\rho})(f_0) = - \delta f_{\ell}(f_0),
\]
and therefore
\[
D\nabla E(f_0) \cdot \left(\begin{array}{c} \delta f_1(f_0) \\ \\ \vdots \\ \\ \delta f_{n-1}(f_0) \end{array}\right) = 2 \, \left(\begin{array}{c} \delta f_1(f_0) \\ \\ \vdots \\ \\ \delta f_{n-1}(f_0) \end{array}\right),
\]
which proves the lemma.
\end{proof}
\vs\vs

%%%%%%%%%%%%%%%%%%%%%%%%%%%%%%%%%%%%%%%%%%%%%%%%%%%%%%%%%%%%%%%%%%%%
%
%
%%%%%%%%%%%%%%%%%%%%%%%%%%%%%%%%%%%%%%%%%%%%%%%%%%%%%%%%%%%%%%%%%%%%
 								% Assembling the solution %
%%%%%%%%%%%%%%%%%%%%%%%%%%%%%%%%%%%%%%%%%%%%%%%%%%%%%%%%%%%%%%%%%%%%
%
%
% %%%%%%%%%%%%%%%%%%%%%%%%%%%%%%%%%%%%%%%%%%%%%%%%%%%%%%%%%%%%%%%%%%%

\subsection{Assembling the solution} 
	\label{ss:ass}
It remains to piece together the extremal disks constructed locally above into a global solution of the HCMA equation. By compactness of $D$, there exists an $\epsilon > 0$ small enough so that corollary \ref{cor:H} applies for all $z' \in D$ and all $\lambda$ such that $|\lambda| < \epsilon$. Given a $z_0 \in D$, we can set up our special coordinates near $z_0$ so that, e.g., $z'(z_0) = 0$, and by corollary \ref{cor:H}, if we set $f = f(\zeta; z', \lambda) = H(z', \lambda)(\zeta)$, with $\lambda$ fixed, gives us a smooth map 
\[
\Psi_{\lambda, 0}: U(z_0) \times \D \to \u_{\lambda}, \; (z', \zeta) \to f(\zeta; z', \lambda)
\]
which extends to be $\mathcal{C}^{k,\alpha}$ on $U(z_0) \times \bar{\D}$. Here $U(z_0)$ is a convenient open neighborhood (in $D$) of $z_0 \in D$. This map is a diffeomorphism onto a neighborhood of $f(\bar{\D}; 0, \lambda)$ for $U(z_0)$ and $\lambda$ small enough, since this is true for $\lambda = 0$. Explicitly, in our special coordinates, for $\lambda = 0$, $\u_0 = \{(z', t, \lambda = 0) \, | \; h(z')|t|^2 < 1\}$, and the extremal disk is given by
\begin{equation}
	\label{eqn:ext0}
f(\zeta; z', \lambda = 0) = (z', \frac{\zeta}{h^{\frac12}(z')}, \lambda = 0).
\end{equation}
To globalize things, let us rescale $\Psi_{0,0}$ and define it from $\u_0$ to itself by 
\[
\Psi_0(z', t, \lambda = 0) = f(h^{\frac12}(z') t; z', \lambda = 0) = (z', t, \lambda = 0),
\]
by (\ref{eqn:ext0}). That is, $\Psi_0 = Id: \u_0 \to \u_0$. Similarly, define $\Psi_{\lambda}: \u_0 \to \u_{\lambda}$ by the local formula
\[
\Psi_{\lambda}(z', t, \lambda) = f(h^{\frac12}(z') t; z', \lambda).
\]
We leave to the reader to check that corollary \ref{cor:H} shows that these local formulas agree on overlaps and give a globally defined map $\Psi_{\lambda}$. The maps $\Psi_{\lambda}$ depend smoothly on $\lambda$ and are smooth in the interiors of $\u_0, \u_{\lambda}.$ They extend to be of class $\mathcal{C}^{k,\alpha}$ from $\bar{\u}_0 \to \bar{\u}_{\lambda}.$ (A simple argument shows they are in fact smooth up to the boundary.) Since $\Psi_0$ is a diffeomorphism, so is $\Psi_{\lambda}$, for $|\lambda|$ small enough. This diffeomorphism is holomorphic along the fibers of $\u_0 \subset  L$, and carries this foliation by fibers of $L$ to the foliation $\f_{\lambda}$ of $\u_{\lambda}$ by the extremal disks. Define $\Phi_{\lambda} = \Psi^{-1}_{\lambda}: \u_{\lambda} \to \u_0$, and define $u_{\lambda} = \Phi_{\lambda}^*u_0$. We claim that $\v_{\lambda}, u_{\lambda}$ can be taken to be the $\v, u$, respectively, in Theorem \ref{thm:converse}, for $\lambda$ small enough. Since $u_{\lambda}$ is harmonic along the leaves of $\f_{\lambda}$, it suffices to show that $u_{\lambda}$ is plurisubharmonic of rank $n-1$. 
\vs
\begin{lemma}
	\label{lem:upsh}
		For $u_{\lambda} = \Phi^*_{\lambda} u_0$ as above, we have, for $\lambda$ sufficiently small:
		\vs
			\begin{enumerate}
				\item the level sets $\u_{\lambda,c} := \{u_{\lambda} = c \in [0,+\infty)\}$ are smooth and strictly pseudoconvex in $\u_{\lambda}$.
				\vs
				\item for $f:\D \to \u_{\lambda}$. the (1,0) form $\zeta f^*\partial u_{\lambda}$ extends as a holomorphic, non-vanishing section of $f^*T(\u_{\lambda})$ on $\D$.
			\end{enumerate}
\end{lemma}
\vs
\begin{proof} 
	First note that both statements are obviously true for $\lambda = 0$. Since $u_{\lambda}$ restricted to a leaf of $\f_{\lambda}$ is $\log \frac1{|\zeta|^2}$ in the uniformizing parameter, $d\, u_{\lambda} \neq 0$ on $\u_{\lambda} - D$ and the level sets are smooth. The strict pseudoconvexity is equivalent to showing that the Levi matrix
\begin{equation}
	\label{eqn:Levi}
	\mathcal{L} = \mathcal{L}(u) = \left(\begin{array}{ccccc} 0 & u_{\bar{1}} & \cdots & u_{\overline{n-1}} & u_{\bar{t}} \\ u_1 & u_{1,\bar{1}} & \cdots & u_{1,\overline{n-1}} & u_{1,\bar{t}} \\ \vdots & & u_{k, \bar{\ell}} &  & \vdots \\ u_t & u_{t,\overline{1}} & \cdots & u_{t,\overline{n-1}} & u_{t,\bar{t}} \end{array} \right), k, \ell = 1, \ldots, n-1,
\end{equation}
for $u = u_{\lambda}$ be non-singular on $\u_{\lambda} - D$. This is because the level sets of $u_0$ on $\u_0 - D$ are strictly pseudoconvex, and all level sets of $u_{\lambda}$ would have non-degenerate Levi forms, which could not change signature with $\lambda$.
\vs
To analyze $\mathcal{L}(u_{\lambda})$, we write it locally in the special coordinates we have been using. Notice that the coordinates will be $(z',t)$ both on $\u_0$ and $\u_{\lambda}$ with the caution that $\u_{\lambda} \subset \c^n \times \c$ with $\u_{\lambda}$ defined by $z_n = \lambda \cdot t.$ Setting $\Phi_{\lambda, n} = \Phi_{\lambda}^* t,$ we have that
\begin{equation}
	\label{eqn:uform}
		u_{\lambda}(z',t) = \log \frac1{h(z') |\Phi_{\lambda, n}(z',t)|^2} = \log \frac1{h(z')} + \log \frac1{|\Phi_{\lambda, n}(z',t)|^2}.
\end{equation}
Note that $\log \frac1{h(z')}$ is strictly plurisubharmonic in $z'$, independently of $\lambda$. To analyze $- \log |\Phi_{\lambda,n}|^2$, we use Taylor's theorem to write on $\u_{\lambda}$
\[
\Phi_{\lambda,n}(z',t) = t A_{\lambda}(z',t) + \bar{t} B_{\lambda}(z', t),
\]
where 
\vskip 1mm
\begin{equation}
	\label{eqn:AB}
		\begin{array}{rl}
		(1) & A_{\lambda}, B_{\lambda} \; \text{are smooth in the parameters} \; z', t, \lambda \on \bar{\u}\\
		 & \\
		(2) & A_0 \equiv 1, \and B_0 \equiv 0
		\end{array}
\end{equation}
\vs
\noindent Property (2) in (\ref{eqn:AB}) follows from the fact that $\Phi_0$ is the identity mapping. It follows from (\ref{eqn:AB}) that all derivatives of positive order of $A_{\lambda}, B_{\lambda}$ are uniformly $O(\lambda)$ on $\u_{\lambda}$, from which the following estimates follow, uniform on $\u$:
\begin{equation}
	\label{eqn:HessP}
		\begin{array}{rl}
		(1) & \Phi_{\lambda,n} = t(1 + O(\lambda)). \\
		& \\
		(2) & (\Phi_{\lambda,n})_k = (\Phi_{\lambda,n})_{\bar{k}} = t O(\lambda), \, k = 1,\ldots, n-1.\\
		&  \\
		(3) & (\Phi_{\lambda,n})_{k,\bar{\ell}} =  t O(\lambda), \, k, \ell = 1,\ldots, n-1. \\
		& \\
		(4) & (\Phi_{\lambda,n})_t = 1 + O(\lambda), (\Phi_{\lambda,n})_{\bar{t}} = t O(\lambda). \\
		& \\
		(5) & (\Phi_{\lambda,n})_{k,\bar{t}} = (\Phi_{\lambda,n})_{t,\bar{k}} = (\Phi_{\lambda,n})_{t,\bar{t}} = O(\lambda), \, k = 1,\ldots, n-1.
		\end{array}
\end{equation}
\noindent Consider the matrix $M$ defined as
\vskip 1mm
\[
M = \left(\begin{array}{ccc}  0 & 0 & 1 \\ 0 & \delta_{k,\ell} & 0 \\ 1 & 0 & t \end{array} \right), \; k, \ell = 1, \ldots, n.
\]
\noindent Using (\ref{eqn:HessP}) we see that
\[
M \mathcal{L}(u_{\lambda}) \bar{M} = \left(\begin{array}{ccccc} 0 & & -(\log h(z'))_{\bar{\ell}} & & 1 \\ & & &  & \\ -(\log h(z'))_{k} & & -(\log h(z'))_{k,\bar{\ell}} & & 0\\ & & & & \\ 1 & & 0 & & 0 \end{array} \right) + O(\lambda).
\]
Hence, there exist constants $C_1 > C_2 > 0$ so that 
\[
- \frac{C_1}{|t|^2} \leq \det \, \mathcal{L} \leq - \frac{C_2}{|t|^2} < 0,
\]
on $\u_{\lambda}$, for $\lambda$ small enough, proving point (1.) of lemma \ref{lem:upsh}.
\vs
To prove point (2.) of the lemma, we use an argument of \cite{lempert81}, p. 462. Let us note first that we could use $u_{\lambda}$ as a defining function for $\partial \u_{\lambda}$ in place of the $r$ used above. Hence, along $\partial f(\D)$ we have $\partial u_{\lambda} = - p(\zeta) \, \partial \, r$, for $p(\zeta), \zeta \in \partial \D$, is real and strictly positive. The following lemma collects facts we have already proved and need now.
\begin{lemma}
	\label{lem:extensions}
		\[
		\begin{array}{rl} (1) & \frac{\zeta}{p \rho} u_{\lambda, t} \, \in \HH^{k,\alpha}.\\
					      &		\\
					(2) & \text{the extension} \, - g \, \text{of} \; \frac{\zeta}{p \rho} u_{\lambda, t} \; \text{to} \, \D \, \text{is non-vanishing.} \\
					     &  \\
					(3) & \frac{u_{\lambda, \ell}}{p \rho} \, \text{extends to} \; g_{\ell}, \; \text{holomorphic on} \, \D \and \in \mathcal{C}^{k,\alpha}(\bar{\D}).     		
		\end{array}
		\]
\end{lemma}
\begin{proof}(of lemma \ref{lem:extensions})
Points (1) and (2) are contained in (\ref{eqn:RHrt}). Part (3) is the content of corollary \ref{cor:rihol}.
\end{proof}
Now consider the one-form 
\begin{equation}
	\label{eqn:beta}
		\beta = g \, dt + \, \zeta \, \sum g_{\ell} \, dz_{\ell},
\end{equation}
defined, holomorphic and nowhere vanishing along $f(\D)$. Along $f(\partial \D)$, 
\[
\beta = - \frac{\zeta}{p \rho} \, \partial u_{\lambda}.
\]
Since $\beta$ is holomorphic, we get $f^*\beta(\dd[\;,\zeta])$ is holomorphic, but since $f^*u_{\lambda} = \log \frac1{|\zeta|^2}$, we get $f^*(\frac{\zeta}{p \rho} \, \partial \, u_{\lambda})(\dd[\;, \zeta]) = - \frac1{p \rho}$, which is real. Hence $f^*\beta(\dd[\;,\zeta])$ is constant on $\bar{\D}$ and $\frac1{p \rho}$ is constant on $\partial \D$. Scaling $\beta$, we can assume $p \rho \equiv 1$.
\vs
\begin{lemma}
	\label{lem:tangency}
		$\beta = - \zeta \, \partial u_{\lambda}$ along $f(\bar{\D})$.
\end{lemma}
\vs
\begin{proof}(of lemma \ref{lem:tangency})
Fix $c \in (0, +\infty)$, and consider the level surface $\u_{\lambda, c} := \{u_{\lambda} = c\}$ in $\u_{\lambda}$. For $z \in \u_{\lambda, c},$ let $v$ be a vector tangent to $\u_{\lambda, c}$ at $z$ which lies in the complex tangent space $\{ v \in T(\u_{\lambda,c}) \, | \, \partial u_{\lambda}(v) = 0\}$ at $z$. Let $\gamma(\sigma)$ be a smooth curve lying in $\u_{\lambda,c}$ with $\gamma(0) = z$ and $\gamma'(0) = v$. Then there are smooth functions $z' = z'(\sigma), \zeta = \zeta(\sigma)$ such that $\gamma(\sigma) = f(\zeta(\sigma), z'(\sigma), \lambda)$, where $f(\zeta, z', \lambda)$ is the extremal function guaranteed by the diffeomorphism $\Psi_{\lambda}$. Thus, $\Phi_{\lambda}(\gamma(\sigma)) = (z'(\sigma), t(\sigma))$ and $\zeta(\sigma) = \frac{t(\sigma)}{h^{\frac12}(z'(\sigma))}.$ Set
\[
\delta f = \dd[\;, \sigma] \, f(\zeta, z'(\sigma), \lambda) \, |_{\sigma = 0}.
\]
This is a holomorphic section of $T(\u_{\lambda})$ along $f(\D)$. Since $f(\zeta, z'(\sigma), \lambda) \in \partial \u_{\lambda}$ for all $\zeta \in \partial \D$, we have that 
\[
\re \, \partial u_{\lambda}(\delta f) \, = \, \re \, \zeta^{-1} \beta(\delta f) = 0.
\]
On the other hand, since $f(0, z'(\sigma), \lambda) \equiv (z'(\sigma), 0) \in D,$ we see that $\delta f(0)$ is tangent to $D$ at $(z'(0), 0)$. Taking account of (\ref{eqn:beta}), we see that the holomorphic function $\beta(\delta f)$ vanishes at $\zeta = 0$. Hence, the function $\zeta^{-1} \beta(\delta f)$ is holomorphic on $\D$. Therefore, $\re \, \zeta^{-1} \, \beta(\delta f)$ is harmonic on $\D$ and vanishes on $\partial \D$ and thus vanishes identically on $\D$. This follows for any $v^{(1,0)} = \delta f (z)$ which is tangent to $\u_{\lambda,c}$. If in addition, $\partial u_{\lambda}(\delta f(z)) = 0$, then $\beta(\delta f(z)) = 0$. Since $\beta$ and $\zeta \partial u_{\lambda}$ are non-zero at $z$, there exists a non-zero complex $a(\zeta)$ for $\zeta \neq 0$ such that 
\[
	\beta = a(\zeta) \zeta \partial u_{\lambda}, \, \text{at} \, f(\zeta, z'(0), \lambda), \zeta \neq 0.
\]
Pulling this identity back to $\D$ and evaluating both sides on $\dd[\;,\zeta]$, we see $a(\zeta) \equiv -1.$
\end{proof}
\vs
The proof of lemma \ref{lem:upsh} part (2) is immediate from lemma \ref{lem:tangency}.
\end{proof}
\vs
We now remark that by lemma \ref{lem:upsh} part (2), the 2-form $i \partial \bar{\partial} u_{\lambda}$ has rank $< n$, but by lemma \ref{lem:upsh} part (1), it is positive definite restricted to the kernel $\{ \partial u_{\lambda} = 0\}$. Hence it is plurisubharmonic of rank exactly $n-1$ on $\u_{\lambda} - D$. which completes the proof of Theorem \ref{thm:converse}.
%
%
%
%
%
%
%%%%%%%%%%%%%%%%%%%%%%%%%%%%%%%%%%%%%%%%%%%%%%%%%%%%%%%%%%%
%%%%%%%%%%%%%%%%%%%%%%%%%%%%%%%%%%%%%%%%%%%%%%%%%%%%%%%%%%%
%																				                 %	
%								       Final comments										         %	
%																				                 %	
%%%%%%%%%%%%%%%%%%%%%%%%%%%%%%%%%%%%%%%%%%%%%%%%%%%%%%%%%%%
%%%%%%%%%%%%%%%%%%%%%%%%%%%%%%%%%%%%%%%%%%%%%%%%%%%%%%%%%%%
 %
 %
 %
 %
 %
 %
 \section{Further questions}\label{sec:final}
 We arrange these final remarks parallel to the preceding sections.
 \vs
Concerning the uniqueness theorem of section \ref{sec:entirezoll}, a similar result would be valid on the other CROSSes whenever the analogue of the Kobayashi-Ochiai, Kachi-Koll\'ar results hold true. Are these true for the CROSSes $\c\p^n, \h\p^n$ and $\mathbb{O}\p^2$, {\em i.e.}, for the corresponding projective complexifications? These (minimal, equivariant) projective complexifications are listed in \cite{pw}.
\vs
The obvious question about section \ref{sec:algsuff} is whether similar projective compactificatons can be constructed when we do not have the very simple structure on the corresponding Monge-Amp\`ere foliation. Similarly, it is unclear whether for a general affine manifold $X$ there exists a strongly p.s.h. exhaustion $\tau$ with $u = \log \tau$ verifying the HCMA equation near infinity. It is known that exhaustions $u$ satisfying the HCMA equation near infinity exist, but these usually do not have $e^{\tau}$ strictly p.s.h. It seems technically important to have this condition to provide a background K\"ahler metric against which to construct holomorphic functions on $X$ with growth conditions. 
\vs
Finally, concerning section \ref{sec:algnec}, for the CROSSes other than $\r\p^n$, there is the solution $u$ of the HCMA equation on the corresponding (entire) Grauert tube. A variational problem can be set up along the divisor $D$ at infinity of the projective complexification in \cite{pw}, as in subsection \ref{ss:vars} here, which applies to a finite dimensional family of competitor disks. There is a unique stable disk through each point in the divisor at infinity. Can one perturb the totally real zero section corresponding to the original CROSS and create a Monge-Amp\`ere foliation and solution on the complement of such a perturbation? Is there a geometric use or interpretation for the holomorphic competitor disks other than the stable ones, which will exist in all cases other than $M = \r\p^n$? Is there a complex analytic interpretation of Zoll geometry directly using such disks, in dimension 2 or hopefully even in higher dimensions?
 
\bibliographystyle{plain}
\bibliography{bibtex}
\end{document}